\documentclass[review]{elsarticle}

\usepackage{lineno,hyperref}
\linenumbers
\modulolinenumbers[1]

\usepackage{framed,multirow}
\usepackage{verbatim}
\usepackage{amssymb}
\usepackage{latexsym}
\usepackage{amsmath}
\usepackage{lipsum}
\usepackage{amsfonts}
\usepackage{graphicx}
\usepackage{epstopdf}
\usepackage{bm,nicefrac}
\usepackage{subfig}
\usepackage{setspace}
\usepackage{algcompatible}
\usepackage{tikz}
\usepackage{pgfplots}
\usepackage{mathptmx}
\usepackage{amsmath}
\usepackage{amssymb}
\usepackage{amsthm}
\usepackage{comment}
\usepackage{algorithm}

\usepackage[noend]{algpseudocode}

\makeatletter
\def\BState{\State\hskip-\ALG@thistlm}
\makeatother

\newtheorem{thm}{Theorem}[section]

\theoremstyle{definition}

\usepackage{url}
\usepackage{xcolor}
\usepackage{comment}
\renewcommand{\hat}{\widehat}
\renewcommand{\tilde}{\widetilde}
\definecolor{newcolor}{rgb}{.8,.349,.1}
\newcommand{\logLogSlopeTriangle}[5]
{
	\pgfplotsextra
	{
		\pgfkeysgetvalue{/pgfplots/xmin}{\xmin}
		\pgfkeysgetvalue{/pgfplots/xmax}{\xmax}
		\pgfkeysgetvalue{/pgfplots/ymin}{\ymin}
		\pgfkeysgetvalue{/pgfplots/ymax}{\ymax}
		
		\pgfmathsetmacro{\xArel}{#1}
		\pgfmathsetmacro{\yArel}{#3}
		\pgfmathsetmacro{\xBrel}{#1-#2}
		\pgfmathsetmacro{\yBrel}{\yArel}
		\pgfmathsetmacro{\xCrel}{\xArel}
		
		\pgfmathsetmacro{\lnxB}{\xmin*(1-(#1-#2))+\xmax*(#1-#2)} 
		\pgfmathsetmacro{\lnxA}{\xmin*(1-#1)+\xmax*#1} 
		\pgfmathsetmacro{\lnyA}{\ymin*(1-#3)+\ymax*#3} 
		\pgfmathsetmacro{\lnyC}{\lnyA+#4*(\lnxA-\lnxB)}
		\pgfmathsetmacro{\yCrel}{\lnyC-\ymin)/(\ymax-\ymin)} 
		
		\coordinate (A) at (rel axis cs:\xArel,\yArel);
		\coordinate (B) at (rel axis cs:\xBrel,\yBrel);
		\coordinate (C) at (rel axis cs:\xCrel,\yCrel);
		
		\draw[#5]   (A)-- node[pos=0.5,anchor=north] {}
		(B)-- 
		(C)-- node[pos=0.5,anchor=west] {\small{#4}}
		cycle;
	}
}

\journal{Journal of Computational Physics}
\graphicspath{{./Figs/}}
\newcommand{\note}[1]{\textcolor{black}{#1}}
\newcommand{\LRp}[1]{\left(#1\right)}

\begin{document}


\begin{frontmatter}

\title{A weight-adjusted discontinuous Galerkin method for wave propagation in coupled elastic-acoustic media }%

\author[1]{Kaihang {Guo}\corref{cor1}}
\author[2]{Sebastian Acosta}
\author[1]{Jesse {Chan}}
\cortext[cor1]{Corresponding author: 
	Email: kaihang.guo@rice.edu;
	Tel.: +1-281-702-8829;  
	}

\address[1]{Department of Computational and Applied Mathematics, Rice University, 6100 Main St, Houston, TX 77005, United States}
\address[2]{Department of Pediatrics-Cardiology, Baylor College of Medicine, Houston, TX, United States}


\begin{abstract}
This paper presents a high-order discontinuous Galerkin (DG) scheme for the simulation of wave propagation through coupled elastic-acoustic media. We use a first-order stress-velocity formulation, and derive a simple upwind-like numerical flux which weakly imposes continuity of the normal velocity and traction at elastic-acoustic interfaces. When combined with easily invertible weight-adjusted mass matrices \cite{chan2017weight, chan2017curved, chan2018weight}, the resulting method is efficient, consistent, and energy stable on curvilinear meshes and for arbitrary heterogeneous media, including anisotropy and sub-cell (micro) heterogeneities.  We numerically verify the high order accuracy and stability of the proposed method, and investigate its performance for applications in photoacoustic tomography. 
\end{abstract}


\end{frontmatter}

 
\section{Introduction}
Simulations of wave propagation through elastic-acoustic coupling media are applicable to a \note{wide} range of scientific and engineering areas.  For example, coupled elastic-acoustic media arises when simulating wave propagation through the human bone and tissue.  While wave propagation in tissue is modeled by the acoustic wave equation, wave propagation in bone is more accurately modeled using the elastic wave equation, and when considering wave propagation through both bone and tissue, careful attention is required for treatment of the elastic-acoustic interface. Wave propagation through coupled elastic-acoustic media also arises in seismology, where oceans are modeled as acoustic materials and the earth is modeled as an elastic medium.  

Several high order finite element methods have been developed for coupled acoustic-elastic wave propagation based on both first and second order formulations of the underlying equations.  In \cite{komatitsch2000wave}, Komatitsch et al.\ use a spectral element method (SEM) for the second order form of the equations, and enforce the coupling between acoustic and elastic media using with a predictor-multicorrector iteration at each time step.  A more efficient time-stepping approach based on explicit coupling conditions is proposed in \cite{chaljub2004spectral, komatitsch2005spectral}.  Discontinuous Galerkin (DG) methods have also been developed for coupled acoustic-elastic media, with elastic-acoustic interface conditions typically incorporated through modifications of the numerical flux.  For second order equations, Antonietti et al\note{.}\  \cite{antonietti2018high} analyze the stability and convergence of a symmetric interior penalty DG formulation on polygonal and polyhedral meshes.  Appelo and Wang \cite{appelo2018energy} introduce an ``energy-based'' second order DG method which can be made to either conserve or dissipate energy based on the choice of numerical flux.  

DG methods, which were originally developed for first order hyperbolic equations, are also widely used for first-order formulations.  \note{Wilcox et al. \cite{wilcox2010high} derive an upwind numerical flux from the exact Riemann problem at acoustic-acoustic, elastic-elastic, and elastic-acoustic interfaces, and use this to construct a first-order velocity-strain DG-SEM scheme for coupled isotropic elastic-acoustic media on meshes of curved hexahedral elements. The authors show stability of the continuous DG formulation; however, semi-discrete stability in the presence of inexact quadrature, curved meshes, and sub-cell heterogeneities is not discussed in detail.} 
\note{In \cite{zhan2018exact}, Zhan et al. extend this DG-SEM method to anisotropic elastic-acoustic media by solving a simplified Riemann problem on inter-element interfaces, though high order accuracy and energy stability are not addressed theoretically.} Ye et al\note{.} \cite{ye2016discontinuous} circumvent the Riemann problem altogether by using a DG formulation with a dissipative upwind-like ``penalty'' flux.  The resulting DG method is high order accurate and provably energy stable for anisotropic elastic-acoustic media with piecewise constant heterogeneities.  

In this paper, we \note{develop} a high order DG method for acoustic-elastic media based on the first order stress-velocity form of the equations.  The proposed method utilizes a simple dissipative upwind-like penalty flux and weight-adjusted mass matrices (a generalization of mass lumping) \cite{chan2017weight,chan2017curved,chan2018weight}.  The method applicable to unstructured and curved tetrahedral meshes, and is high order accurate and energy stable in the presence of arbitrary heterogeneous media including anisotropy and micro (sub-cell) \note{heterogeneities}. \note{Instead of an exact upwind flux, we add upwind-like dissipation through a penalty flux based on natural continuity conditions between acoustic-acoustic, elastic-elastic, and coupled acoustic-elastic interfaces. Like the upwind flux, the penalty flux adds dissipation and achieves theoretically optimal high order convergence rates for all numerical experiments without impacting the maximum time-step size. However, expressions for the penalty flux are significantly simpler than the fluxes developed by Wilcox et al.\ and Zhan et al.\ \cite{wilcox2010high,zhan2018exact}. Additionally, we prove that the penalty flux is consistent and that the semi-discrete DG formulation is energy stable for general ``modal'' DG formulations in the presence of both sub-cell heterogeneities and curved elements. Experiments with high order DG discretizations on curvilinear simplicial meshes verify these theoretical properties.} 



The outline of the paper is as follows: In Section~\ref{sec:pre}, we review DG formulations for the acoustic and elastic wave equations.  In Section~\ref{sec:acouselasdg}, we introduce the numerical flux for elastic-acoustic interfaces and prove that the resulting DG formulation is energy stable and consistent.  In Section~\ref{sec:numerical}, we verify the stability and accuracy of the proposed DG method, and conclude in Section~\ref{sec:app} with an application in photoacoustic tomography (PAT).  

\section{Weight-adjusted DG methods for acoustic and elastic wave propagation}\label{sec:pre}

In this section, we briefly review high order DG discretizations for the acoustic and elastic wave equations.  In the presence of micro (sub-cell) heterogeneities, inverse weighted mass matrices appear in the matrix forms of these discretizations.  These inverses are approximated using easily invertible \textit{weight-adjusted} mass matrices, resulting in a weight-adjusted DG method.  The weight-adjusted approach will be extended to the DG formulation for coupled elastic-acoustic wave propagation and curvilinear meshes in \note{Sections~\ref{sec:eawave}} and \ref{sec:curvedDG}.

\subsection{Mathematical notation}\label{sec:notation}
We assume a physical domain $\Omega$, which is exactly represented by a triangulation $\Omega_h$ consisting of $K$ non-overlapping elements $D^k$.  We assume that each element $D^k$ is the image of the reference element $\widehat{D}$ under a mapping $\bm{\Phi}^k$ 
$$\bm{x}=\bm{\Phi}^k\widehat{\bm{x}},\qquad \bm{x}\in D^k,\quad \widehat{\bm{x}}\in \widehat{D},$$
where $\bm{x}=\left(x,y,z\right)$ are physical coordinates on the $k$th element and $\widehat{\bm{x}}=\left(\hat{x},\hat{y},\hat{z}\right)$ are coordinates on the reference element. Over each element $D^k$, we define the polynomial approximation space $V_h\left(D^k\right)$ as
$$\note{V_h\left(D^k\right)= V_h\left(\widehat{D}\right)\circ \left(\bm{\Phi}^k\right)^{-1}=\{ \hat{v}_h \circ (\Phi^k)^{-1}, \  \hat{v}_h \in V_h(\hat{D}),\},}$$
where $V_h\left(\widehat{D}\right)$ is a polynomial approximation space of degree $N$ on the reference element.  In this work,\footnote{In three dimensions, the reference element and approximation space are the bi-unit right tetrahedron and total degree $N$ polynomials
\[
\widehat{D}=\{\left( \hat{x},\hat{y},\hat{z}\right)\geq -1,\quad \hat{x}+\hat{y}+\hat{z}\leq -1\}, \qquad V_h\left(\widehat{D}\right)=P^N\left(\widehat{D}\right)=\big\{\hat{x}^i \hat{y}^j \hat{z}^k,\quad 0\leq i+j+k\leq N\big\}.
\]
}
the reference element is taken to be bi-unit right triangle,
$$\widehat{D}=\{\left( \hat{x},\hat{y}\right)\geq -1,\quad \hat{x}+\hat{y}\leq 0\},$$
and the reference approximation space $V_h\left(\widehat{D}\right)$ is taken to be total degree $N$ polynomials,
$$V_h\left(\widehat{D}\right)=P^N\left(\widehat{D}\right)=\big\{\hat{x}^i \hat{y}^j,\quad 0\leq i+j\leq N\big\}.$$

\subsection{Discontinuous Galerkin methods for first-order wave equations}\label{sec:DGmethod}
On an element $D^k$, we define the jump of scalar and vector valued functions across element interfaces as
$$[\![p]\!]=p^+-p, \qquad [\![\bm{u}]\!]=\bm{u}^+-\bm{u},$$
where $p^+, \bm{u}^+$ and $p, \bm{u}$ are the neighboring and local traces of the solution over the interface, respectively. \note{Note that, for a shared interface between two elements $D^k$ and $D^{k,+}$, the sign of the jump is different depending on whether the jump is defined with respect to $D^k$ or $D^{k,+}$.}  The average across an interface is defined as
$$\{\!\{p\}\!\}=\frac{1}{2}\left(p^++p\right),\qquad \{\!\{\bm{u}\}\!\}=\frac{1}{2}\left(\bm{u}^++\bm{u}\right).$$ 

In this work, we use a first-order pressure-velocity formulation for the acoustic wave equation (e.g.\ in fluid media)
\begin{equation}
\begin{split}
\frac{1}{\rho c^2}\frac{\partial p}{\partial t}=\nabla\cdot\bm{u},\\
\rho \frac{\partial\bm{u}}{\partial t}=\nabla p,
\end{split}
\label{eq:awave}
\end{equation}
where $p$ is the acoustic pressure, $\bm{u}\in \mathbb{R}^d$ is the vector of velocities in each coordinate direction, and $\rho$ and $c$ are density and wavespeed, respectively. For simplicity, we assume unit density $\rho=1$. We also assume that (\ref{eq:awave}) is posed over time $t\in[0,T)$ on the physical domain $\Omega$ with boundary $\partial\Omega$, with the wavespeed bounded from above and below by
$$0<c_{\textmd{min}}\leq c(\bm{x})\leq c_{\textmd{max}}<\infty.$$

We adopt a DG variational formulation from \cite{warburton2013low}, which is given over element $D^k$ by
\begin{equation}
\begin{split}
\left(\frac{1}{c^2}\frac{\partial p}{\partial t},q\right)_{L^2(D^k)} &= \left(\nabla\cdot\bm{u},q\right)_{L^2(D^k)} + \sum_{f\in \partial D^k}\left\langle \frac{1}{2}\bm{n}^T[\![\bm{u}]\!]+\frac{\tau_p}{2}[\![p]\!],q\right\rangle_{L^2(f)},\\
\left(\frac{\partial \bm{u}}{\partial t},\bm{w}\right)_{L^2(D^k)} &= \left(\nabla p,\bm{w}\right)_{L^2(D^k)} + \sum_{f\in \partial D^k}\left\langle \frac{1}{2}[\![p]\!]\bm{n}+\frac{\tau_u}{2}[\![\bm{u}]\!],\bm{w}\right\rangle_{L^2(f)},
\end{split}
\label{eq:acousdg}
\end{equation}
where $\bm{n}$ is the \note{outward} normal vector, and $\tau_p,\tau_u$ are penalty parameters.  Here, $\left(u,v\right)_{L^2\left(D^k\right)}$ and ${\left\langle u,v\right\rangle}_{L^2(f)}$ denote the $L^2$ inner products over $D^k$ and a face $f$ of the surface $\partial D^k$, respectively.  

For the elastic wave equation, we use a symmetrized first-order stress-velocity formulation from \cite{chan2018weight}. Let $\rho$ be the density and $\bm{C}$ be the symmetric matrix form of constitutive tensor relating stress and strain. The first-order system in $d$ dimensions is given by
\begin{equation}
\begin{split}
\rho\frac{\partial\bm{v}}{\partial t}&=\sum_{i=1}^{d}\bm{A}_i^T\frac{\partial\bm{\sigma}}{\partial\bm{x}_i},\\
\bm{C}^{-1}\frac{\partial\bm{\sigma}}{\partial t}&=\sum_{i=1}^{d} \bm{A}_i\frac{\partial\bm{v}}{\partial\bm{x}_i},
\end{split}
\label{eq:ewave}
\end{equation}
where $\bm{v}$ is the vector of velocity and $\bm{\sigma}$ is a vector consisting of unique entries of the symmetric stress tensor. In two dimensions, \note{the  matrices} $\bm{A}_i$ are defined as 
$$\bm{A}_1=
\renewcommand\arraystretch{1}
\begin{pmatrix}
1&0&0\\
0&0&0\\
0&0&1
\end{pmatrix},\qquad
\bm{A}_2=
\begin{pmatrix}
0&0&0\\
0&1&0\\
1&0&0
\end{pmatrix},
$$
and the expression of $\bm{A}_i$ in three dimensions can be found in \cite{chan2018weight}.  \note{Note that, by factoring out $\bm{C}$, the resulting matrices $A_i$ do not involve any material coefficients and all entries are either $0$ or $1$.}  The elastic wave equation is discretized using the following DG formulation:
\begin{equation}
\begin{split}
\left(\rho\frac{\partial\bm{v}}{\partial t},\bm{w}\right)_{L^2(D^k)}&=\left(\sum^d_{i=1}\bm{A}_i^T\frac{\partial\bm{\sigma}}{\partial \bm{x}_i},\bm{w}\right)_{L^2(D^k)}+\sum_{f\in \partial D^k}\left\langle \frac{1}{2}\bm{A}_n^T[\![\bm{\sigma}]\!]+\frac{\tau_v}{2}\bm{A}_n^T\bm{A}_n[\![\bm{v}]\!],\bm{w}\right\rangle_{L^2(f)},\\
\left(\bm{C}^{-1}\frac{\partial\bm{\sigma}}{\partial t},\bm{q}\right)_{L^2(D^k)}&=\left(\sum^d_{i=1}\bm{A}_i\frac{\partial\bm{v}}{\partial \bm{x}_i},\bm{q}\right)_{L^2(D^k)}+\sum_{f\in \partial D^k}\left\langle \frac{1}{2}\bm{A}_n[\![\bm{v}]\!]+\frac{\tau_{\sigma}}{2}\bm{A}_n\bm{A}_n^T[\![\bm{\sigma}]\!],\bm{q}\right\rangle_{L^2(f)},
\end{split}
\label{eq:elasdg}
\end{equation}
where $\bm{A}_n$ is normal matrix defined as $\bm{A}_n=\sum^d_{i=1}\bm{n}_i\bm{A}_i$, and terms $\tau_v,\tau_{\sigma}$ are penalty parameters introduced on element interfaces.  The DG formulations (\ref{eq:acousdg}) and (\ref{eq:elasdg}) are provably consistent and energy stable for non-negative penalty parameters $\tau_p, \tau_u, \tau_v,\tau_{\sigma}\geq 0$ \cite{chan2017weight,chan2018weight}.

\subsection{The semi-discrete matrix system}\label{sec:matrixsystem}
The matrix form of the DG formulations in the previous section involve mass and differentiation matrices. We assume the reference and physical approximation spaces $V_h\left(\hat{D}\right)$ and $V_h\left(D^k\right)$ are spanned by bases $\{\phi_i\}_{i=1}^{N_p}$ and $\{\phi_i^k\}_{i=1}^{N_p}$, respectively.  The mass matrix $\bm{M}^k$, weighted mass matrix $\bm{M}^k_w$ and face mass matrix $\bm{M}^k_f$ for $D^k$ are defined as
\begin{equation*}
\begin{split}
\left(\bm{M}^k\right)_{ij} &=\int_{D^k}\phi^k_j\phi^k_i= \int_{\hat{D}}\phi_j\phi_i J^k,\\ 
\left(\bm{M}^k_{w}\right)_{ij} &= \note{\int_{D^k} w\phi^k_j\phi^k_i}= \note{\int_{\hat{D}} w\phi_j\phi_i J^k},\\
\left(\bm{M}_f^k\right)_{ij} &=\int_{\partial D^k_f}\phi^k_j\phi^k_i = \int_{\partial\hat{D}_f}\phi_j\phi_i J_f^k,
\end{split}
\end{equation*}
where $J^k$ and $J^k_f$ are the volume and face Jacobian of the affine mapping $\bm{\Phi}^k$, and $w(\bm{x})$ is a spatially varying positive and bounded weight. We also define weak differentiation matrices $\bm{S}_i$ with entries
$$\left(\bm{S}_1\right)_{ij}=\int_{\hat{D}}\frac{\partial\phi_j}{\partial x}\phi_iJ^k,\qquad \left(\bm{S}_2\right)_{ij}=\int_{\hat{D}}\frac{\partial\phi_j}{\partial y}\phi_iJ^k,\qquad \left(\bm{S}_3\right)_{ij}=\int_{\hat{D}}\frac{\partial\phi_j}{\partial z}\phi_iJ^k.$$
Using the above notation, the DG formulation (\ref{eq:acousdg}) can be written in matrix form as
\begin{equation*}
\begin{split}
\bm{M}^k_{1/c^2}\frac{d\bm{p}}{dt}&=\sum_{j=1}^d\bm{S}_j^k\bm{U}_j+\sum_{f=1}^{N_{\textmd{faces}}}\bm{M}^k_fF_p\left(\bm{p},\bm{p}^+,\bm{U},\bm{U}^+\right),\\
\bm{M}^k\frac{d\bm{U}_i}{dt} &=\bm{S}^k_i\bm{p}+\sum_{f=1}^{N_{\textmd{faces}}}\bm{n}_i\bm{M}^k_fF_u\left(\bm{p},\bm{p}^+,\bm{U},\bm{U}^+\right),\qquad i=1,\dots,d,
\end{split}
\end{equation*}
where $\bm{U}_i$ and $\bm{p}$ are degrees of freedom for $\bm{u}_i$ and $p$. The flux terms $F_p,F_u$ are defined such that
\begin{equation*}
\begin{split}
\left(\bm{M}^k_fF_p\left(\bm{p},\bm{p}^+,\bm{U},\bm{U}^+\right)\right)_j&=\int_{\partial D^k_f}\frac{1}{2}\left(\tau_p[\![p]\!]+\bm{n}\cdot[\![\bm{u}]\!]\right)\phi^k_j,\\
\left(\bm{n}_i\bm{M}^k_fF_u\left(\bm{p},\bm{p}^+,\bm{U},\bm{U}^+\right)\right)_j&=\int_{\partial D^k_f}\frac{1}{2}\left(\tau_u[\![\bm{u}]\!]\cdot\bm{n}+[\![p]\!]\right)\phi^k_j\bm{n}_i.
\end{split}
\end{equation*}

The DG scheme (\ref{eq:elasdg}) for the elastic wave equations can similarly be written as
\begin{equation*}
\begin{split}
\bm{M}^k_{\rho\bm{I}}\frac{d\bm{V}}{dt}&=\sum_{i=1}^d\left(\bm{A}_i^T\otimes\bm{S}_i^k\right)\bm{\Sigma}+\sum_{f=1}^{N_{\textmd{faces}}}\left(\bm{I}\otimes\bm{M}^k_f\right)F_v,\\
\bm{M}^k_{\bm{C}^{-1}}\frac{d\bm{\Sigma}}{dt}&=\sum_{i=1}^d\left(\bm{A}_i\otimes\bm{S}_i^k\right)\bm{V}+\sum_{f=1}^{N_{\textmd{faces}}}\left(\bm{I}\otimes\bm{M}^k_f\right)F_{\sigma},
\end{split}
\end{equation*}
where $F_v, F_{\sigma}$ denote the elastic flux terms, $\otimes$ denotes the Kronecker product, and the matrix-valued weight mass matrix $\bm{M}^k_{\bm{C}^{-1}}$ is defined as
\[
\bm{M}^k_{\bm{C}^{-1}} = \begin{bmatrix}
\bm{M}^k_{\bm{C}^{-1}_{11}} & \ldots &\bm{M}^k_{\bm{C}^{-1}_{1d}}\\
\vdots & \ddots & \vdots\\
\bm{M}^k_{\bm{C}^{-1}_{d1}} & \ldots &\bm{M}^k_{\bm{C}^{-1}_{dd}}\\
\end{bmatrix}, 
\]
where $\bm{C}^{-1}_{ij}$ denotes the $ij$th entry of $\bm{C}^{-1}$ and $\bm{M}^k_{\bm{C}^{-1}_{ij}}$ denotes the scalar weighted mass matrix with weight $\bm{C}^{-1}_{ij}$.

\subsection{Weight-adjusted discontinuous Galerkin method}\label{sec:WADG}
In this work, we pair high order DG methods with explicit time-stepping schemes, which require the inversion of DG mass matrices at each time-step.  Let $\bm{U}$ denote the vector of all DG degrees of freedom, and let $\bm{M}_{1/c^2}^k, \bm{A}_k$ denote the local matrices representing the local DG mass mass matrix and spatial DG formulation, such that the semi-discrete DG scheme can be written over $D^k$ as follows:
\begin{equation}
\frac{d\bm{U}}{dt}=\left(\bm{M}_{1/c^2}^k\right)^{-1}\bm{A}_k\bm{U}.
\label{eq:wadginvert}
\end{equation}

When the wavespeed $c^2$ is approximated by a constant over each element, it is possible to apply $\LRp{\bm{M}^k_{1/c^2} }^{-1}$ using only the constant values of $J^k, c^2$ over each element and a single reference mass matrix inverse $\bm{M}^{-1}$ over the entire mesh. 
However, inverses of weighted mass matrices are distinct from element to element when $c^2$ possesses sub-element variations.  Typical implementations precompute and store these weighted mass matrix inverses \cite{mercerat2015nodal, bencomo2015discontinuous}, which significantly increases the storage cost of high order DG schemes.  

To address this issue, we use a weight-adjusted discontinuous Galerkin (WADG) is proposed in \cite{chan2017weight,chan2018weight}, which is energy stable and high order accurate for sufficiently regular weighting functions. WADG approximates each weighted mass matrix by a weight-adjusted approximation $\widetilde{\bm{M}}^k_w$
$$\bm{M}^k_w\approx\widetilde{\bm{M}}_w^k=\bm{M}^k\left(\bm{M}^k_{1/w}\right)^{-1}\bm{M}^k.$$
The inverse of $\widetilde{\bm{M}}^k_w$ is then
\begin{equation}
\left(\bm{M}^k_w\right)^{-1}\approx \left(\widetilde{\bm{M}}^k_w\right)^{-1} = \left(\bm{M}^k\right)^{-1}\bm{M}^k_{1/w}\left(\bm{M}^k\right)^{-1}.
\label{eq:WADGscalar}
\end{equation}
Since the weight only appears in $\bm{M}^k_{1/w}$, $\left(\widetilde{\bm{M}}^k_w\right)^{-1}$ can be applied using reference inverse mass matrices and a matrix-free quadrature-based evaluation of $\bm{M}^k_{1/w}$.  
Analogously, the inverse of $\bm{M}_{\bm{C}^{-1}}$ can be approximated by the inverse of a matrix-weighted weight-adjusted mass matrix 
\begin{equation*}
\bm{M}^{-1}_{\bm{C}^{-1}}\approx \LRp{\bm{I}\otimes \bm{M}^{-1}} \bm{M}_{\bm{C}} \LRp{\bm{I}\otimes \bm{M}^{-1}}.
\end{equation*}

In practice, weight-adjusted mass matrix inverses are applied in a matrix-free fashion using sufficiently accurate quadrature rules. We follow \cite{chan2017weight} and use simplicial quadratures which are exact for polynomials of degree $2N+1$ \cite{xiao2010numerical}.  Let $\widehat{\bm{x}}_i,\widehat{\bm{w}}_i$ denote the quadrature points and weights on the reference element $\hat{D}$. We define the interpolation matrix $\bm{V}_q$ as
$$\left(\bm{V}_q\right)_{ij}=\phi_j\left(\hat{\bm{x}}_i\right),$$
whose columns consist of values of basis functions at quadrature points. On each element $D^k$, we have 
\begin{equation*}
\bm{M}^k=J^k\bm{M}=J^k\bm{V}_q^T\textmd{diag}\left(\hat{\bm{w}}\right)\bm{V}_q,\ \ \ 
\bm{M}^k_{c^2}=J^k\bm{V}_q^T\textmd{diag}\left(\bm{d}\right)\bm{V}_q,\ \ \  \bm{d}_i=\frac{\hat{\bm{w}}_i}{c^2\left(\bm{\Phi}^k\hat{\bm{x}}_i\right)}
\end{equation*}
where $\bm{\Phi}^k\hat{\bm{x}}_i$ are quadrature points on $D^k$ and \note{ $c^2\left(\bm{\Phi}^k\hat{\bm{x}}\right)$ denote the values of the wavespeed at quadrature points}.  
Plugging the approximation (\ref{eq:WADGscalar}) into the local DG formulation (\ref{eq:wadginvert}), we obtain
\begin{equation}
\frac{d\bm{U}}{dt}=\left(\bm{M}^k\right)^{-1}\bm{M}^k_{c^2}\left(\bm{M}^k\right)^{-1}\bm{A}_k\bm{U}.
\label{eq:WADGmatrix}
\end{equation}
Evaluating $\left(\bm{M}^k\right)^{-1}\bm{A}_k\bm{U}$ is equivalent to the evaluation of the DG right hand side for a unit weight $1/c^2 = 1$.  
Evaluating the remainder of the right hand side of (\ref{eq:WADGmatrix}) requires applying the product of an unweighted mass matrix and weighted mass matrix. This can be done using quadrature-based matrices as follows:
\begin{equation}
\left(\bm{M}^k\right)^{-1}\bm{M}^k_{c^2} = \bm{P}_q\textmd{diag}\left(\frac{1}{c^2\left(\bm{\Phi}^k\hat{\bm{x}}\right)}\right)\bm{V}_q,
\label{eq:pwadg}
\end{equation}
where $\bm{P}_q=\bm{M}^{-1}\bm{V}_q^T\textmd{diag}\left(\hat{\bm{w}}\right)$ is a quadrature discretization of the polynomial $L^2$ projection operator on the reference element. Moreover, since $\bm{P}_q, \bm{V}_q$ are reference operators, the implementation of (\ref{eq:pwadg}) requires only $O\left(N^d\right)$ storage for values of the wavespeed $c^2\left(\bm{\Phi}^k\hat{\bm{x}}\right)$ at quadrature points for each element.  In contrast, storing full weighted mass matrix inverses or factorizations requires $O\left(N^{2d}\right)$ storage on each element.  For example, in three dimensions, the number of quadrature points on one element scales with $O(N_p)=O(N^3)$, while number of entries in each weighted mass matrix inverse is $O(N_p)\times O(N_p)$, implying an $O(N^6)$ storage requirement.

\section{Discontinuous Galerkin methods for coupled elastic-acoustic wave equations}\label{sec:acouselasdg}

For the first-order acoustic and elastic wave equations, the discontinuous Galerkin schemes (\ref{eq:acousdg}) and (\ref{eq:elasdg}) are consistent and discretely energy stable for a large class of quadrature rules. The goal of this work is to extend these existing schemes to solve wave problems in elastic-acoustic coupled media.  The challenge is to derive an appropriate numerical flux for the interface between acoustic and elastic domains. In this section, we propose a new numerical flux across elastic-acoustic interfaces, and prove the consistency and discrete energy stability of the elastic-acoustic DG formulation under this new flux.  

\subsection{Upwind-like numerical flux}\label{sec:eawave} We begin with the continuity conditions on the interface between different media.  For an acoustic-acoustic interface, the normal velocity and pressure are continuous, i.e.,
\begin{equation*}
\bm{u}^+\cdot\bm{n}=\bm{u}\cdot\bm{n},\qquad p^+ = p.
\label{eq:acouscondition}
\end{equation*}
 For an elastic-elastic interface, the velocity and the traction are continuous, i.e.,
\begin{equation*}
\bm{v}^+=\bm{v},\qquad \bm{A}_n^T\bm{\sigma}^+=\bm{A}_n^T\bm{\sigma}.
\label{eq:elascondition}
\end{equation*}
For an interface 
between elastic and acoustic media, the normal component of the velocity and the traction are continuous, i.e.,
\begin{equation}
\bm{u}\cdot\bm{n}=\bm{v}\cdot\bm{n},\qquad \bm{A}_n^T\bm{\sigma} = p\bm{n},
\label{eq:condition}
\end{equation}
where $\bm{u}$ and $\bm{v}$ denote velocity in acoustic and elastic media, respectively. Based on these continuity conditions, we derive an upwind-like numerical flux for the elastic-acoustic interface. 

For clarity, we will distinguish between acoustic and elastic fluxes at a coupled elastic-acoustic interface.  Let $\Omega_e, \Omega_a$ denote the elastic and acoustic domains, respectively.  Let  $\textcolor{black}{\Gamma_{ea}}$ and  $\textcolor{black}{\Gamma_{ae}}$ denote the respective boundaries of $\Omega_e$ and $\Omega_a$ which correspond to the acoustic-elastic interface.  On $\textcolor{black}{\Gamma_{ae}}$, the numerical fluxes are taken to be
\begin{equation*}
\begin{split}
\frac{1}{2}\bm{n}^T\left(\bm{v}-\bm{u}\right)  +\frac{\tau_p}{2}\bm{n}^T\left(\bm{A}_n^T\bm{\sigma}-p\bm{n}\right)\qquad&\textmd{(pressure)},\\  \frac{1}{2}\bm{n}\bm{n}^T\left(\bm{A}_n^T\bm{\sigma}-p\bm{n}\right)+\frac{\tau_u}{2}\bm{n}\bm{n}^T\left(\bm{v}-\bm{u}\right)\qquad&\textmd{(velocity)},
\end{split}
\end{equation*} 
while the numerical fluxes on $\textcolor{black}{\Gamma_{ea}}$ are given by
\begin{equation*}
\begin{split}
\frac{1}{2}\bm{A}_n\bm{n}\bm{n}^T(\bm{u}-\bm{v})+\frac{\tau_{\sigma}}{2}\bm{A}_n(p\bm{n}-\bm{A}^T_n\bm{\sigma})\qquad&\textmd{(stress)}, \\ \frac{1}{2}\left(p\bm{n}-\bm{A}_n^T\bm{\sigma}-(\bm{I}-\bm{n}\bm{n^T})\bm{A}^T_n\bm{\sigma}\right)+\frac{\tau_v}{2}\bm{n}\bm{n}^T(\bm{u}-\bm{v})\qquad&\textmd{(velocity)}.
\end{split}
\end{equation*}
We now formulate a DG scheme for the first-order elastic-acoustic coupled wave equations.  In the acoustic domain $\Omega_a$, the DG formulation is given by
\begin{equation}
\begin{split}
\left(\frac{1}{c^2}\frac{\partial p}{\partial t},q\right)_{L^2(D^k)} =& \left(\nabla\cdot\bm{u},q\right)_{L^2(D^k)} + \sum_{f\in \partial D^k\cap \textcolor{black}{\Gamma_{aa}}}\left\langle \frac{1}{2}\bm{n}^T[\![\bm{u}]\!]+\frac{\tau_p}{2}[\![p]\!],q\right\rangle_{L^2(f)}\\[1ex]&+\sum_{f\in \partial D^k\cap \textcolor{black}{\Gamma_{ae}}}\left\langle \frac{1}{2}\bm{n}^T\left(\bm{v}-\bm{u}\right)  +\frac{\tau_p}{2}\bm{n}^T\left(\bm{A}_n^T\bm{\sigma}-p\bm{n}\right),q\right\rangle_{L^2(f)}\\[2ex]
\left(\frac{\partial \bm{u}}{\partial t},\bm{w}\right)_{L^2(D^k)} = &\left(\nabla p,\bm{w}\right)_{L^2(D^k)} + \sum_{f\in \partial D^k\cap \textcolor{black}{\Gamma_{aa}}}\left\langle \frac{1}{2}[\![p]\!]\bm{n}+\frac{\tau_u}{2}[\![\bm{u}]\!],\bm{w}\right\rangle_{L^2(f)}\\[1ex]&+\sum_{f\in \partial D^k\cap \textcolor{black}{\Gamma_{ae}}}\left\langle \frac{1}{2}\bm{n}\bm{n}^T\left(\bm{A}_n^T\bm{\sigma}-p\bm{n}\right)+\frac{\tau_u}{2}\bm{n}\bm{n}^T\left(\bm{v}-\bm{u}\right),\bm{w}\right\rangle_{L^2(f)}.
\end{split}
\label{eq:acouselas}
\end{equation}
In the elastic domain $\Omega_e$, the DG formulation is given by
\begin{equation}
\begin{split}
\left(\rho\frac{\partial\bm{v}}{\partial t},\bm{w}\right)_{L^2(D^k)}&=\left(\sum^d_{i=1}\bm{A}_i^T\frac{\partial\bm{\sigma}}{\partial \bm{x}_i},\bm{w}\right)_{L^2(D^k)}+\sum_{f\in \partial D^k\cap \textcolor{black}{\Gamma_{ee}}}\left\langle \frac{1}{2}\bm{A}_n^T[\![\bm{\sigma}]\!]+\frac{\tau_v}{2}\bm{A}_n^T\bm{A}_n[\![\bm{v}]\!],\bm{w}\right\rangle_{L^2(f)}\\[1ex]&+\sum_{f\in \partial D^k\cap \textcolor{black}{\Gamma_{ea}}}\left\langle \frac{1}{2}\left(p\bm{n}-\bm{A}_n^T\bm{\sigma}-(\bm{I}-\bm{n}\bm{n^T})\bm{A}^T_n\bm{\sigma}\right)+\frac{\tau_v}{2}\bm{n}\bm{n}^T(\bm{u}-\bm{v}),\bm{w}\right\rangle_{L^2(f)},\\[2ex]
\left(\bm{C}^{-1}\frac{\partial\bm{\sigma}}{\partial t},\bm{q}\right)_{L^2(D^k)}&=\left(\sum^d_{i=1}\bm{A}_i\frac{\partial\bm{v}}{\partial \bm{x}_i},\bm{q}\right)_{L^2(D^k)}+\sum_{f\in \partial D^k\cap \textcolor{black}{\Gamma_{ee}}}\left\langle \frac{1}{2}\bm{A}_n[\![\bm{v}]\!]+\frac{\tau_{\sigma}}{2}\bm{A}_n\bm{A}_n^T[\![\bm{\sigma}]\!],\bm{q}\right\rangle_{L^2(f)}\\[1ex]&+\sum_{f\in \partial D^k\cap \textcolor{black}{\Gamma_{ea}}}\left\langle \frac{1}{2}\bm{A}_n\bm{n}\bm{n}^T(\bm{u}-\bm{v})+\frac{\tau_{\sigma}}{2}\bm{A}_n(p\bm{n}-\bm{A}^T_n\bm{\sigma}),\bm{q}\right\rangle_{L^2(f)}.
\end{split}
\label{eq:elasacous}
\end{equation}

We note that media heterogeneities are incorporated into the left hand side of the DG formulations (\ref{eq:acouselas}) and (\ref{eq:elasacous}), and that the numerical fluxes are independent of any variations in $1/c^2, \bm{C}^{-1}$.  In our numerical experiments, we approximate the weighted mass matrices induced by micro (sub-cell) heterogeneities in $1/c^2, \bm{C}^{-1}$ by easily invertible weight-adjusted mass matrices as described in Section~\ref{sec:WADG}.

\subsection{Consistency and energy stability}

In this section, we prove that the DG formulations (\ref{eq:acouselas}) and (\ref{eq:elasacous}) are consistent and energy stable in arbitrary heterogeneous media.  

\begin{thm} The coupled discontinuous Galerkin scheme is consistent.  
\end{thm}
\begin{proof}
	Assume that $\bm{u},p,\bm{v},\bm{\sigma}$ are exact solutions of coupled elastic-acoustic wave equations, and that boundary conditions are imposed through consistent modifications of the numerical flux.\footnote{The stable and consistent imposition of boundary conditions is described in \cite{chan2017weight,chan2018weight}.}. Then, plugging them into (\ref{eq:acouselas}) and (\ref{eq:elasacous}) causes the volume terms to vanish.  Consistency follows if the numerical flux terms also vanish.  
	
	At acoustic-acoustic interfaces, the pressure and normal velocity are continuous.  Thus, the numerical flux reduces to
	$$ \frac{1}{2}\bm{n}^T[\![\bm{u}]\!]+\frac{\tau_p}{2}[\![p]\!]=0,\qquad \frac{1}{2}[\![p]\!]\bm{n}+\frac{\tau_u}{2}[\![\bm{u}]\!]=0.$$
	At elastic-elastic interfaces, the traction $\bm{A}_n^T\bm{\sigma}$ and the velocity are continuous, and the numerical flux reduces to
	$$\frac{1}{2}\bm{A}_n^T[\![\bm{\sigma}]\!]+\frac{\tau_v}{2}\bm{A}_n^T\bm{A}_n[\![\bm{v}]\!]=0,\qquad \frac{1}{2}\bm{A}_n[\![\bm{v}]\!]+\frac{\tau_{\sigma}}{2}\bm{A}_n\bm{A}_n^T[\![\bm{\sigma}]\!]=0.$$
	For an elastic-acoustic interface $\textcolor{black}{\Gamma_{ae}}$ ,we have 
	\begin{equation*}
	\begin{split}
	\frac{1}{2}\bm{n}^T\left(\bm{v}-\bm{u}\right)  +\frac{\tau_p}{2}\bm{n}^T\left(\bm{A}_n^T\bm{\sigma}-p\bm{n}\right)= \frac{\tau_p}{2}\bm{n}^T\left(p\bm{n}-p\bm{n}\right)=0,\\[1ex]
	\frac{1}{2}\bm{n}\bm{n}^T\left(\bm{A}_n^T\bm{\sigma}-p\bm{n}\right)+\frac{\tau_u}{2}\bm{n}\bm{n}^T\left(\bm{v}-\bm{u}\right)=\frac{1}{2}\bm{n}\bm{n}^T\left(p\bm{n}-p\bm{n}\right)=0.
	\end{split}
	\end{equation*}
	Similarly, on $\textcolor{black}{\Gamma_{ea}}$, we have
	\begin{equation*}
	\begin{split}
	\frac{1}{2}\left(p\bm{n}-\bm{A}_n^T\bm{\sigma}-(\bm{I}-\bm{n}\bm{n^T})\bm{A}^T_n\bm{\sigma}\right)+\frac{\tau_v}{2}\bm{n}\bm{n}^T(\bm{u}-\bm{v})=\frac{1}{2}\left(\bm{I}-\bm{n}\bm{n}^T\right)p\bm{n}=0,\\
	\frac{1}{2}\bm{A}_n\bm{n}\bm{n}^T(\bm{u}-\bm{v})+\frac{\tau_{\sigma}}{2}\bm{A}_n(p\bm{n}-\bm{A}^T_n\bm{\sigma})=0.
	\end{split}
	\end{equation*}
	Thus, consistency holds for acoustic-acoustic, elastic-elastic and elastic-acoustic interfaces, which implies the coupled DG scheme is consistent.
\end{proof}

The formulations (\ref{eq:acouselas}) and (\ref{eq:elasacous}) can also be shown to be energy stable for any choice of $\tau_u=\tau_v\geq0, \tau_p=\tau_{\sigma}\geq 0$. For simplicity, we assume zero homogeneous Dirichlet boundary conditions on $\partial \Omega$ in the proof of energy stability. 
\begin{thm}
	The coupled discontinuous Galerkin scheme is energy stable for $\tau_u=\tau_v\geq0, \tau_p=\tau_{\sigma}\geq0$, in the sense that
	\note{
	\begin{equation*}
	\begin{split}
	&\sum_{D^k\in\Omega_h^a}\frac{\partial}{\partial t}\left(\left(\frac{p}{c^2},p\right)_{L^2(D^k)}+\left(\bm{u},\bm{u}\right)_{L^2(D^k)}\right)+
	\sum_{D^k\in\Omega_h^e}\frac{\partial}{\partial t}\left(\left(\rho\bm{v},\bm{v}\right)_{L^2(D^k)}+\left(\bm{C}^{-1}\bm{\sigma},\bm{\sigma}\right)_{L^2(D^k)}\right)\\
	=&-\sum_{f\in \Gamma_{aa}}\int_f\left(\frac{\tau_p}{2}[\![p]\!]^2+\frac{\tau_u}{2}\left(\bm{n}\cdot[\![\bm{u}]\!]\right)^2\right)\note{\mathop{d\bm{x}}}-\sum_{f\in \Gamma_{ee}}\int_f\left(\frac{\tau_u}{2}|\bm{A}_n[\![\bm{v}]\!]|^2+\frac{\tau_p}{2}|\bm{A}_n^T[\![\bm{\sigma}]\!]|^2\right)\note{\mathop{d\bm{x}}}\\&-\sum_{f\in\Gamma_{ea}\cup\Gamma_{ae}}\int_{f}\left(\frac{\tau_u}{2}|\bm{n}^T(\bm{u}-\bm{v})|^2+\frac{\tau_p}{2}|p\bm{n}-\bm{A}_n^T\bm{\sigma}|^2\right)\note{\mathop{d\bm{x}}}\leq 0,
	\end{split}
	\end{equation*}}
	where $\Omega_h^a$ and $\Omega_h^e$ denote the acoustic and elastic computational domain, respectively.
	\label{thm:stable}
\end{thm}
\begin{proof}
For the acoustic part, taking $q=p,\ \bm{w}=\bm{u}$ and integrating the divergence term of the pressure equation by parts gives
\begin{equation}
\begin{split}
\left(\frac{1}{c^2}\frac{\partial p}{\partial t},p\right)_{L^2(D^k)} =& -\left(\nabla p,\bm{u}\right)_{L^2(D^k)} + \sum_{f\in \partial D^k\cap \textcolor{black}{\Gamma_{aa}}}\left\langle \frac{1}{2}\bm{n}^T\{\!\{\bm{u}\}\!\}+\frac{\tau_p}{2}[\![p]\!],p\right\rangle_{L^2(f)}\\[1ex]&+\sum_{f\in \partial D^k\cap \textcolor{black}{\Gamma_{ae}}}\left\langle \frac{1}{2}\bm{n}^T\left(\bm{v}-\bm{u}\right)  +\frac{\tau_p}{2}\bm{n}^T\left(\bm{A}_n^T\bm{\sigma}-p\bm{n}\right),p\right\rangle_{L^2(f)}\\[2ex]
\left(\frac{\partial \bm{u}}{\partial t},\bm{u}\right)_{L^2(D^k)} = &\left(\nabla p,\bm{u}\right)_{L^2(D^k)} + \sum_{f\in \partial D^k\cap \textcolor{black}{\Gamma_{aa}}}\left\langle \frac{1}{2}[\![p]\!]\bm{n}+\frac{\tau_u}{2}[\![\bm{u}]\!],\bm{u}\right\rangle_{L^2(f)}\\[1ex]&+\sum_{f\in \partial D^k\cap \textcolor{black}{\Gamma_{ae}}}\left\langle \frac{1}{2}\bm{n}\bm{n}^T\left(\bm{A}_n^T\bm{\sigma}-p\bm{n}\right)+\frac{\tau_u}{2}\bm{n}\bm{n}^T\left(\bm{v}-\bm{u}\right),\bm{u}\right\rangle_{L^2(f)}.
\end{split}
\label{eq:strong-weak}
\end{equation}
Adding the pressure and velocity equations together and summing over all element $D^k$ gives
\begin{equation*}
\begin{split}
&\sum_{D^k\in\Omega_h^a}\frac{\partial}{\partial t}\left(\left(\frac{p}{c^2},p\right)_{L^2(D^k)}+\left(\bm{u},\bm{u}\right)_{L^2(D^k)}\right)\\=&-\note{\frac{1}{2}}\sum_{f\in \Gamma_{aa}}\int_f\left(\tau_p[\![p]\!]^2+\tau_u\left(\bm{n}\cdot[\![\bm{u}]\!]\right)^2\right)\note{\mathop{d\bm{x}}}
\\&+\frac{1}{2}\sum_{f\in \Gamma_{ae}}\int_f \left(\bm{u}^T\bm{n}\bm{n}^T\bm{A}_n^T\bm{\sigma}+p\bm{v}^T\bm{n}+\tau_v\bm{u}^T\bm{n}\bm{n}^T\left(\bm{v}-\bm{u}\right)+\tau_pp\bm{n}^T\left(\bm{A}^T_n\bm{\sigma}-p\bm{n}\right)\right)\note{\mathop{d\bm{x}}}
\end{split}
\label{eq:acousenergy}
\end{equation*}

For the elastic part, taking $\bm{q}=\bm{\sigma},\ \bm{w}=\bm{v}$ and Theorem 3.1 in \cite{chan2018weight} gives
\begin{equation*}
\begin{split}
&\sum_{D^k\in\Omega_h^e}\frac{\partial}{\partial t}\left(\left(\rho\bm{v},\bm{v}\right)_{L^2(D^k)}+\left(\bm{C}^{-1}\bm{\sigma},\bm{\sigma}\right)_{L^2(D^k)}\right)\\
=&-\note{\frac{1}{2}}\sum_{f\in \Gamma_{ee}}\int_f\left(\tau_u|\bm{A}_n[\![\bm{v}]\!]|^2+\tau_p|\bm{A}_n^T[\![\bm{\sigma}]\!]|^2\right)\note{\mathop{d\bm{x}}}\\
&+\frac{1}{2}\sum_{f\in \Gamma_{ea}}\int_f \left(\bm{u}^T\bm{n}\bm{n}^T\bm{A}_n^T\bm{\sigma}+p\bm{v}^T\bm{n}+\tau_v\bm{v}^T\bm{n}\bm{n}^T\note{\left(\bm{u}-\bm{v}\right)}+\tau_p\bm{\sigma}^T\bm{A}_n\left(p\bm{n}-\bm{A}_n^T\bm{\sigma}\right)\right)\note{\mathop{d\bm{x}}}.
\end{split}
\label{eq:elasenergy}
\end{equation*}

We first consider the case $\tau_u=\tau_p=0$, which corresponds to a non-dissipative central flux.  Then, adding \note{together contributions from integrals on both $\Gamma_{ae}$ and $\Gamma_{ea}$} and consolidating terms involving normal vectors and normal matrices yields
\begin{equation*}
\begin{split}
&\frac{1}{2}\sum_{f\in\Gamma_{ae}}\int_f \left(\bm{u}^T\bm{n}\bm{n}^T\bm{A}_n^T\bm{\sigma}+p\bm{v}^T\bm{n}\right)\note{\mathop{d\bm{x}}} +\frac{1}{2}\sum_{f\in\Gamma_{ea}}\int_f\left(\bm{u}^T\bm{n}\bm{n}^T\bm{A}_n^T\bm{\sigma}+p\bm{v}^T\bm{n}\right)\note{\mathop{d\bm{x}}}\\
=&\frac{1}{2}\sum_{f\in\Gamma_{ae}}\int_f \left(\bm{u}^T\bm{n}\bm{n}^T\bm{A}_n^T\bm{\sigma}+p\bm{v}^T\bm{n}\right)\note{\mathop{d\bm{x}}} +\frac{1}{2}\sum_{f\in\note{\Gamma_{ae}}}\int_f\left(-\bm{u}^T\bm{n}\bm{n}^T\bm{A}_n^T\bm{\sigma}-p\bm{v}^T\bm{n}\right)\note{\mathop{d\bm{x}}}\\
=&\frac{1}{2}\sum_{f\in\note{\Gamma_{ae}}}\int_f \left(\bm{u}^T\bm{n}\bm{n}^T\bm{A}_n^T\bm{\sigma}+p\bm{v}^T\bm{n} -\bm{u}^T\bm{n}\bm{n}^T\bm{A}_n^T\bm{\sigma}-p\bm{v}^T\bm{n}\right)\note{\mathop{d\bm{x}}}=0.
\end{split}
\end{equation*}
Thus, the contribution from the central portion of the flux sums to zero. Next, we can compute the contribution of penalty fluxes for $\tau_u, \tau_p > 0$
\begin{equation*}
\begin{split}
&\frac{1}{2}\sum_{f\in\Gamma_{ae}}\note{\int_f}\left(\tau_u\bm{u}^T\bm{n}\bm{n}^T\left(\bm{v}-\bm{u}\right)+\tau_pp\bm{n}^T\left(\bm{A}^T_n\bm{\sigma}-p\bm{n}\right)\right)\note{\mathop{d\bm{x}}} \\
&+\frac{1}{2}\sum_{f\in\Gamma_{ea}}\int_f\left(\tau_u\bm{v}^T\bm{n}\bm{n}^T\note{\left(\bm{u}-\bm{v}\right)}+\tau_p\bm{\sigma}^T\bm{A}_n\left(p\bm{n}-\bm{A}_n^T\bm{\sigma}\right)\right)\note{\mathop{d\bm{x}}}\\
=&\frac{1}{2}\sum_{f\in\Gamma_{ae}}\int_f\left(\tau_u\bm{u}^T\bm{n}\bm{n}^T\left(\bm{v}-\bm{u}\right)+\tau_pp\bm{n}^T\left(\bm{A}^T_n\bm{\sigma}-p\bm{n}\right)\right)\note{\mathop{d\bm{x}}}\\
&+\note{\frac{1}{2}\sum_{f\in\Gamma_{ae}}\int_f}\left(\tau_u\bm{v}^T\bm{n}\bm{n}^T(\bm{u}-\bm{v})+\tau_p\bm{\sigma}^T\bm{A}_n\left(p\bm{n}-\bm{A}_n^T\bm{\sigma}\right)\right)\note{\mathop{d\bm{x}}}\\
=&\frac{1}{2}\sum_{f\in\Gamma_{ae}}\int_f\left( -\tau_u\left(\bm{u}-\bm{v}\right)^T\bm{n}\bm{n}^T\note{\left(\bm{u}-\bm{v}\right)}+2\tau_pp\bm{n}^T\bm{A}_n^T\bm{\sigma}-\tau_pp\bm{n}^T\bm{n}p-\tau_p\bm{\sigma}^T\bm{A}_n\bm{A}_n^T\bm{\sigma}\right)\note{\mathop{d\bm{x}}}\\
=&-\frac{1}{2}\sum_{f\in\Gamma_{ae}}\int_f\left( \tau_u|\bm{n}^T\left(\bm{u}-\bm{v}\right)|^2+\tau_p|p\bm{n}-\bm{A}^T_n\bm{\sigma}|^2\right)\note{\mathop{d\bm{x}}}\leq 0.
\end{split}
\end{equation*}
Summing all the contributions, we obtain the desired inequality
 \begin{equation*}
 \begin{split}
\frac{\partial}{\partial t} &\sum_{D^k\in\Omega_h^e}\left(\left(\rho\bm{v},\bm{v}\right)_{L^2(D^k)}+\left(\bm{C}^{-1}\bm{\sigma},\bm{\sigma}\right)_{L^2(D^k)}\right)+\sum_{D^k\in\Omega_h^a}\left(\left(\frac{p}{c^2},p\right)_{L^2(D^k)}+\left(\bm{u},\bm{u}\right)_{L^2(D^k)}\right)\\
 =&-\sum_{f\in \Gamma_{aa}}\int_f\left(\note{\frac{\tau_p}{2}}[\![p]\!]^2+\note{\frac{\tau_u}{2}}\left(\bm{n}\cdot[\![\bm{u}]\!]\right)^2\right)\note{\mathop{d\bm{x}}}-\sum_{f\in \Gamma_{ee}}\int_f\left(\frac{\tau_u}{2}|\bm{A}_n[\![\bm{v}]\!]|^2+\frac{\tau_p}{2}|\bm{A}_n^T[\![\bm{\sigma}]\!]|^2\right)\note{\mathop{d\bm{x}}}\\&-\sum_{f\in\note{\Gamma_{ae}}}\int_{f}\left(\frac{\tau_u}{2}|\bm{n}^T(\bm{u}-\bm{v})|^2+\frac{\tau_p}{2}|p\bm{n}-\bm{A}_n^T\bm{\sigma}|^2\right)\note{\mathop{d\bm{x}}}\leq 0.
 \end{split}
 \end{equation*}

\end{proof}

\subsection{Extension to curvilinear meshes}\label{sec:curvedDG}

The stability of the DG formulations (\ref{eq:acouselas}) and (\ref{eq:elasacous}) in Theorem~\ref{thm:stable} requires the use of integration by parts.  In order to ensure that this same stability holds at the semi-discrete level, integration by parts must hold when integrals are approximated using quadrature.  For affinely mapped simplicial meshes, the geometric terms are constant over each element, such that all spatial integrands on the right-hand side of (\ref{eq:acouselas}) and (\ref{eq:elasacous}) are degree $2N-1$ polynomials. Thus, any quadrature which is exact for at least degree $2N-1$ polynomials is sufficient for stability.  

However, numerous numerical studies demonstrate that, for curved domain boundaries, the use of affinely mapped simplicial meshes limits accuracy to second order\cite{wang2010discontinuous,zhang2015simple,zhang2016curved,Hesthaven2007}.  In this section, we assume the triangulation $\Omega_h$ consists of (possibly curved) elements $D^k$.  Under this assumption, the mapping $\Phi^k$ is no longer affine and the geometric terms are non-constant polynomials within each element.  The resulting spatial integrands in (\ref{eq:acouselas}) and (\ref{eq:elasacous}) are now degree $4N-3$ polynomials, while the surface integrands are degree $4N-2$ polynomials. Thus, the strength of quadrature required to ensure semi-discrete energy stability of the formulations (\ref{eq:acouselas}) and (\ref{eq:elasacous}) is significantly higher for curved meshes than for affine meshes.

We sidestep these quadrature accuracy requirements on curvilinear meshes by using a ``strong-weak'' DG formulation, where we discretize the intermediate DG formulation (\ref{eq:strong-weak}) in Theorem~\ref{thm:stable}.  Similar formulations have been used to guarantee stability under non-standard basis functions \cite{warburton2013low,chan2018multi,KOZDON2019483}.  Because the formulation (\ref{eq:strong-weak}) has already been integrated by parts, the proof of energy stability does not require integrals to be exactly evaluated using quadrature. \note{This quadrature-agnostic stability avoids instability and spurious solution growth for under-integrated DG discretizations on curved meshes \cite{kopriva2016geometry}. }
 However, it does require an explicit quadrature-based discretization, as opposed to a quadrature-free discretization \cite{Hesthaven2007,atkins1998quadrature}.  

We outline the matrices involved in a quadrature-based DG discretization in the following section.  For simplicity, we now assume constant wavespeed $c=1$, such that the strong-weak formulation for the acoustic wave equation is given by
\begin{equation}
\begin{split}
\int_{D^k}\frac{1}{c^2}\frac{\partial p}{\partial t}q&=-\int_{D^{k}}\bm{u}\cdot\nabla q+\int_{\partial D^k}\frac{1}{2}\left(\{\!\{\bm{u}\}\!\}\cdot\bm{n}+\tau_p[\![p]\!]\right)q,\\
\int_{D^k}\frac{\partial \bm{u}}{\partial t}\cdot\bm{w}&=\int_{D^{k}}\nabla p\cdot\bm{w}+\int_{\partial D^k}\frac{1}{2}\left([\![p]\!]+\tau_u[\![\bm{u}]\!]\cdot\bm{n}\right)\bm{w}\cdot\bm{n}.
\end{split}
\label{eq:swaw}
\end{equation}
The mass matrix $\bm{M}^k$ is replaced by a weighted mass matrix with weight $J^k$, which we approximate using a weight-adjusted approximation, i.e.\
$$\left(\bm{M}^k\right)^{-1}\bm{A}^k_h\bm{U}=\bm{M}^{-1}\bm{M}_{1/J^k}\bm{M}^{-1}\bm{A}^k_h\bm{U}.$$
Now, we consider the volume contribution in the pressure equation, i.e.
$$\int_{D^k}\bm{u}\cdot\nabla q=\int_{\hat{D}}\left(\bm{u}_1\frac{\partial q}{\partial x}+\bm{u}_2\frac{\partial q}{\partial y}+\bm{u}_3\frac{\partial q}{\partial z}\right)J^k.$$
This contribution becomes more involved to evaluate due to the face that derivatives now lie on the pressure test function $q$. We follow \cite{chan2017curved,warburton2013low} and evaluate this contribution as
$$\left(\bm{V}_q^{\hat{x}}\right)^T\bm{U}_q^{\hat{x}}+\left(\bm{V}_q^{\hat{y}}\right)^T\bm{U}_q^{\hat{y}}+\left(\bm{V}_q^{\hat{z}}\right)^T\bm{U}_q^{\hat{z}},$$
where $\left(\bm{V}_q^{\hat{x}}\right)^T$ are quadrature-based differentiation matrices defined by
$$\left(\bm{V}_q^{\hat{x}}\right)_{ij}=\frac{\partial \phi_j}{\partial \hat{x}}(\bm{x}_i),\qquad i=1,\dots,N_q.$$
The terms $\bm{U}_q^{\hat{\bm{x}}_i}$ are defined at quadrature points as
\begin{equation*}
\begin{split}
\bm{U}_q^{\hat{x}}&=\textmd{diag}\left(\bm{J}_q\right)\left(\textmd{diag}(\bm{x}_{\hat{x}})\bm{V}_q\bm{U}_1+\textmd{diag}\left(\bm{y}_{\hat{x}}\right)\bm{V}_q\bm{U}_2+\textmd{diag}\left(\bm{z}_{\hat{x}}\right)\bm{V}_q\bm{U}_3\right),\\
\bm{U}_q^{\hat{y}}&=\textmd{diag}\left(\bm{J}_q\right)\left(\textmd{diag}(\bm{x}_{\hat{y}})\bm{V}_q\bm{U}_1+\textmd{diag}\left(\bm{y}_{\hat{y}}\right)\bm{V}_q\bm{U}_2+\textmd{diag}\left(\bm{z}_{\hat{y}}\right)\bm{V}_q\bm{U}_3\right),\\
\bm{U}_q^{\hat{z}}&=\textmd{diag}\left(\bm{J}_q\right)\left(\textmd{diag}(\bm{x}_{\hat{z}})\bm{V}_q\bm{U}_1+\textmd{diag}\left(\bm{y}_{\hat{z}}\right)\bm{V}_q\bm{U}_2+\textmd{diag}\left(\bm{z}_{\hat{z}}\right)\bm{V}_q\bm{U}_3\right),
\end{split}
\end{equation*}
where $\bm{x}_{\hat{x}},\dots$ are evaluations of geometric factors at quadrature points and $\bm{U}_i$ denotes the vector of degrees of freedom for the $i$th velocity component $\bm{u}_i$.  The surface contributions are treated similarly. 


\section{Numerical experiments}\label{sec:numerical}
In this section, we demonstrate the high order convergence and geometric flexibility of the proposed method. In Section~\ref{sec:spectra}, we verify that the semi-discrete scheme is energy stable by computing the spectra of the proposed DG schemes.  In Section~\ref{sec:classical}, we test our method on several classical interface problems with known analytical solutions.  In Section~\ref{sec:curved}, we implement the proposed scheme on curvilinear meshes and perform convergence analyses.  In all numerical experiments, we always choose penalty parameters such that  $\tau_u=\tau_v$ and $\tau_p=\tau_\sigma$.  

\subsection{Spectra and choice of penalty parameter}\label{sec:spectra}
We first verify the energy stability of the proposed method for arbitrary heterogeneous media. We follow the approach in \cite{chan2018weight} and construct a random stiffness matrix using similarity transforms, such that at every quadrature point, $\bm{C}\left(\bm{x}\right)=\bm{U}\bm{D}\bm{U}^T$, where $\bm{D}$ is diagonal matrix with random positive entries $d_{min}\leq \bm{D}_{ii}\leq d_{max}$ and $\bm{U}$ is a random unitary matrix.  For the wavespeed in the acoustic media, we generate positive random values $c_{min}\leq c(\bm{x})\leq c_{max}$ at quadrature nodes.

Let $\bm{L}$ denote the matrix induced by the global semi-discrete DG formulation, such that the time evolution of the global solution is governed by 
$$\frac{\partial\bm{Q}}{\partial t}=\bm{L}\bm{Q}$$
with $\bm{Q}$ denotes a vector of degrees of freedom for $(\bm{u},p,\bm{v},\bm{\sigma})$.  Figure~\ref{fig:affinespectra} shows computed eigenvalues of $L$ for different penalty parameters under discretization parameters $N=3$ and $h=1/4$. In both cases, the largest real part of any eigenvalue is $O(10^{-14})$, verifying that the proposed methods are energy stable up to machine precision.  

For practical simulations, taking $\tau_u,\tau_p>0$ results in damping of under-resolved spurious components of the solution.  However, a naive selection of these parameters can result in a more restrictive time-step restriction for stability. We wish to choose $\tau_u,\tau_p$ as large as possible without increasing the spectra of $\bm{L}$ when using a central flux (i.e. $\tau_u=\tau_p=0$). In Figure~\ref{fig:affinespectra}, we observe that the spectra of $\bm{L}$ for a central flux is roughly half as large as the spectra of $\bm{L}$ when taking $\tau_u=\tau_p=1$. We note that the growth in spectra is due to the large negative real part of the extremal eigenvalues of $\bm{L}$, which consistent with the observation that a subset of eigenvalue of $\bm{L}$ approach $-\infty$ as the penalty parameters increase\cite{chan2017penalty}. Moreover, when we take $\tau_u=\tau_p=0.5$, the largest real part and imaginary part are most have the same magnitude, which indicates that we can add a dissipative term without shortening the time-step size.  

\begin{figure}
	\setcounter{subfigure}{0}
	\subfloat[$\tau_u=\tau_p=0$]{\includegraphics[width=0.35\linewidth]{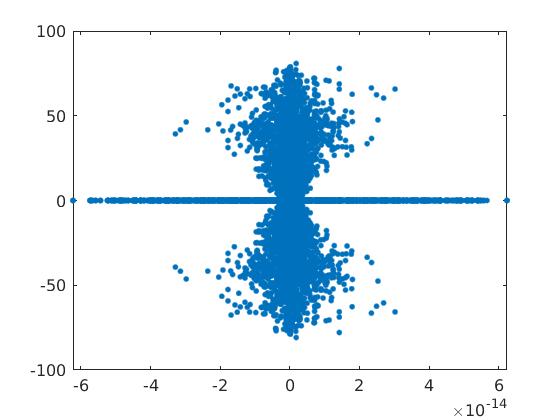}}
	\subfloat[$\tau_u=\tau_p=\frac{1}{2}$]{\includegraphics[width=0.35\linewidth]{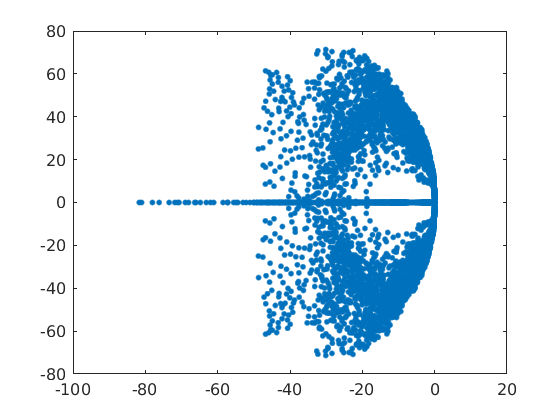}}
	\subfloat[$\tau_u=\tau_p=1$]{\includegraphics[width=0.35\linewidth]{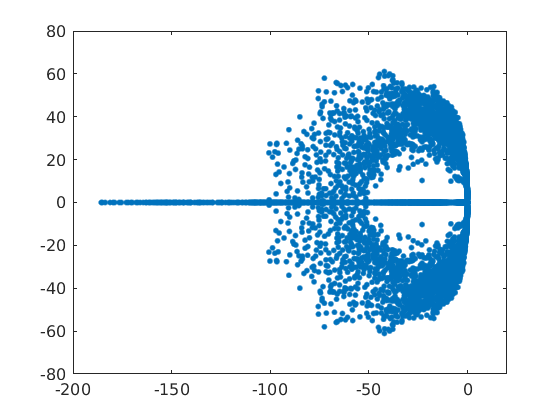}}
\caption{Spectra for $N=3$ on a non-curved uniform mesh with $h=1/4$.  For all cases, the largest real part of the spectra is $O(10^{-14})$.}
\label{fig:affinespectra}
\end{figure}

\subsection{Classical interface problems}\label{sec:classical}

In the following section, we show that the proposed DG method exhibits high order convergence for two classical interface problems: Snell's law and the Scholte wave.  

\subsubsection{Snell's law for an elastic-acoustic interface}
In this experiment, we study the convergence rate of the proposed method for the Snell's law, which models a pressure plane wave incident to an acoustic-elastic interface. The incident wave is reflected as a pressure wave in the acoustic media and transmitted as longitudinal and transverse waves in the elastic media. We follow the problem setting given in \cite{wilcox2010high}. For an incident displacement wave of the form,
$$\bm{w}_{ip}\left(\bm{x},t\right)=C_{ip}\bm{d}_{ip}\cos\left(\kappa_{p1} \left[x_1\sin\left(\alpha_{ip}\right)+x_2\cos\left(\alpha_{ip}\right)\right]-\omega t\right),$$
the reflected displacement wave is 
$$\bm{w}_{rp}\left(\bm{x},t\right)=C_{rp}\bm{d}_{rp}\cos\left(\kappa_{p1} \left[x_1\sin\left(\alpha_{rp}\right)-x_2\cos\left(\alpha_{rp}\right)\right]-\omega t\right).$$
The transmitted longitudinal displacement wave is 
$$\bm{w}_{tp}\left(\bm{x},t\right)=C_{tp}\bm{d}_{tp}\cos\left(\kappa_{p2} \left[x_1\sin\left(\alpha_{tp}\right)+x_2\cos\left(\alpha_{tp}\right)\right]-\omega t\right),$$
and the transmitted transverse displacement wave is 
$$\bm{w}_{ts}\left(\bm{x},t\right)=C_{ts}\bm{d}_{ts}\cos\left(\kappa_{s2} \left[x_1\sin\left(\alpha_{ts}\right)+x_2\cos\left(\alpha_{ts}\right)\right]-\omega t\right).$$
Here, $\omega$ is the angular frequency; $\kappa_{p1}$, $\kappa_{p2}$, and $\kappa_{s2}$ are wavenumbers of the respective waves and $\alpha_{ip}$, $\alpha_{rp}$, $\alpha_{tp}$ and $\alpha_{ts}$ are the associated propagation angles. The displacement directions are
$$\bm{d}_{ip}=\begin{pmatrix}\sin\left(\alpha_{ip}\right)\\\cos\left(\alpha_{ip}\right)\end{pmatrix},\quad \bm{d}_{rp}=\begin{pmatrix}\sin\left(\alpha_{rp}\right)\\-\cos\left(\alpha_{rp}\right)\end{pmatrix},\quad \bm{d}_{tp}=\begin{pmatrix}\sin\left(\alpha_{tp}\right)\\\cos\left(\alpha_{tp}\right)\end{pmatrix},\quad
\bm{d}_{ts}=\begin{pmatrix}-\cos\left(\alpha_{ts}\right)\\\sin\left(\alpha_{ts}\right)\end{pmatrix}.$$
The overall displacement can be written as 
$$\bm{u}\left(\bm{x},t\right)=\begin{cases} \bm{w}_{ip}\left(\bm{x},t\right)+\bm{w}_{rp}\left(\bm{x},t\right), &\textmd{if}\  x_2<0,\\
\bm{w}_{tp}\left(\bm{x},t\right)+\bm{w}_{ts}\left(\bm{x},t\right), &\textmd{otherwise}. \end{cases}$$
The wave speeds in each layer are given by
$$c_{p1}=\sqrt{\frac{\lambda_1+2\mu_1}{\rho_1}},\qquad c_{p2}=\sqrt{\frac{\lambda_2+2\mu_2}{\rho_2}},\qquad c_{s2}=\sqrt{\frac{\mu_2}{\rho_2}},$$
and the corresponding wavenumbers can be computed from the angular frequency
$$\kappa_{p1}=\frac{\omega}{c_{p1}},\qquad \kappa_{p2}=\frac{\omega}{c_{p2}},\qquad \kappa_{s2}=\frac{\omega}{c_{s2}}.$$
Through Snell's Law, the propagation angles are related to the incident angle $\alpha_{ip}$
$$\frac{\sin\left(\alpha_{ip}\right)}{c_{p1}}=\frac{\sin\left(\alpha_{rp}\right)}{c_{p1}}=\frac{\sin\left(\alpha_{tp}\right)}{c_{p2}}=\frac{\sin\left(\alpha_{ts}\right)}{c_{s2}}.$$
The amplitudes of the reflected and transmitted waves are related to the incident wave amplitude
\begin{equation*}
\begin{split}
C_{rp}&=C_{ip} \frac{Z_{p2}\left(\cos\left(2\alpha_{ts}\right)\right)^2+Z_{s2}\left(\sin\left(2\alpha_{ts}\right)\right)^2-Z_{p1}}{Z_{p2}\left(\cos\left(2\alpha_{ts}\right)\right)^2+Z_{s2}\left(\sin\left(2\alpha_{ts}\right)\right)^2+Z_{p1}},\\
C_{tp}&=C_{ip}\frac{c_{p1}\rho_1}{c_{p2}\rho_2} \frac{2Z_{p2}\cos\left(2\alpha_{ts}\right)}{Z_{p2}\left(\cos\left(2\alpha_{ts}\right)\right)^2+Z_{s2}\left(\sin\left(2\alpha_{ts}\right)\right)^2+Z_{p1}},\\
C_{ts}&=C_{ip}\frac{c_{p1}\rho_1}{c_{s2}\rho_2} \frac{2Z_{s2}\sin\left(2\alpha_{ts}\right)}{Z_{p2}\left(\cos\left(2\alpha_{ts}\right)\right)^2+Z_{s2}\left(\sin\left(2\alpha_{ts}\right)\right)^2+Z_{p1}},
\end{split}
\end{equation*}
where 
$$Z_{p1}=\frac{\rho_1c_{p1}}{\cos\left(\alpha_{ip}\right)},\qquad Z_{p2}=\frac{\rho_2c_{p2}}{\cos\left(\alpha_{tp}\right)},\qquad Z_{s2}=\frac{\rho_2c_{s2}}{\cos\left(\alpha_{ts}\right)}.$$
We compute the solution for the specific case of $c_{p1}=1$, $\rho_1=1$, $c_{p2}=3$, $c_{s2}=2$, $\rho_2=1$, $\omega=2\pi$, $\alpha_{ip}=0.2$, and $C_{ip}=1.0$. The computational domain is $[-1,1]^2$ and the exact solution is prescribed by tractions on the boundary. Uniform tetrahedral meshes are used in the experiment.  Figure~\ref{fig:snell} shows the convergence of $L^2$ errors under mesh refinement for both central fluxes and dissipative penalty fluxes.  Optimal $O(h^{N+1})$ rates of convergence are observed for the penalty flux, while an ``odd-even'' convergence pattern is observed for the central flux.  
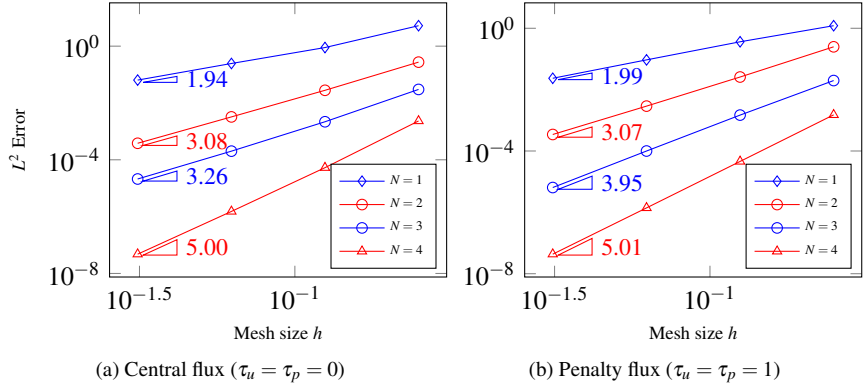
\begin{figure}
\centering
\setcounter{subfigure}{0}
\subfloat[Central flux ($\tau_u=\tau_p=0$)]{
	\begin{tikzpicture}
	\begin{loglogaxis}[
	width=0.5\textwidth,
	title style = {font=\large},
	legend style = {font=\tiny},
	xlabel=Mesh size $h$,
	ylabel=$L^2$ Error,
	xlabel style= {font=\scriptsize},
	ylabel style= {font=\scriptsize},
	legend pos = south east
	]
	\addplot[color=blue,mark=diamond] coordinates {
		(0.25,5.24405)
		(0.125, 0.892497)
		(0.0625,0.245341)
		(0.0312,0.0638811)
	};
	\addplot[color=red,mark=o] coordinates {
		(0.25,0.273655)
		(0.125, 0.0279257)
		(0.0625,0.00323484)
		(0.0312,0.000381534)
	};
	
	\addplot[color=blue,mark=o] coordinates {
		(0.25,0.0299595)
		(0.125, 0.00218157)
		(0.0625,0.000201572)
		(0.0312,0.0000210932)
	};
	\addplot[color=red,mark=triangle] coordinates {
		(0.25,0.0023262)
		(0.125, 0.000052822)
		(0.0625,0.00000152304)
		(0.0312,0.0000000476848)
	};
    \logLogSlopeTriangle{0.2}{0.1}{0.71}{1.94}{blue};
	\logLogSlopeTriangle{0.2}{0.1}{0.48}{3.08}{red};
	\logLogSlopeTriangle{0.2}{0.1}{0.35}{3.26}{blue};
	\logLogSlopeTriangle{0.2}{0.1}{0.08}{5.00}{red};
	\legend{$N=1$,$N=2$,$N=3$,$N=4$}
	\end{loglogaxis}
	\end{tikzpicture}}
\subfloat[Penalty flux ($\tau_u=\tau_p=1$)]{
\begin{tikzpicture}
\begin{loglogaxis}[
width=0.5\textwidth,
title style = {font=\large},
legend style = {font=\tiny},
xlabel=Mesh size $h$,
xlabel style= {font=\scriptsize},
ylabel style= {font=\scriptsize},
legend pos = south east
]
\addplot[color=blue,mark=diamond] coordinates {
	(0.25,1.21227)
	(0.125, 0.36274)
	(0.0625,0.0932777)
	(0.0312,0.0234593)
};
\addplot[color=red,mark=o] coordinates {
	(0.25,0.248666)
	(0.125, 0.0257738)
	(0.0625,0.00290105)
	(0.0312,0.000345199)
};

\addplot[color=blue,mark=o] coordinates {
	(0.25,0.0197668)
	(0.125, 0.00147982)
	(0.0625,0.000099811)
	(0.0312,0.00000644744)
};
\addplot[color=red,mark=triangle] coordinates {
	(0.25,0.00150798)
	(0.125, 0.0000456722)
	(0.0625,0.00000139757)
	(0.0312,0.0000000433868)
};
\logLogSlopeTriangle{0.2}{0.1}{0.72}{1.99}{blue};
\logLogSlopeTriangle{0.2}{0.1}{0.51}{3.07}{red};
\logLogSlopeTriangle{0.2}{0.1}{0.32}{3.95}{blue};
\logLogSlopeTriangle{0.2}{0.1}{0.08}{5.01}{red};
\legend{$N=1$,$N=2$,$N=3$,$N=4$}
\end{loglogaxis}
\end{tikzpicture}}
\caption{Convergence of $L^2$ errors for the Snell's law solution.}
\label{fig:snell}
\end{figure}

\subsubsection{Scholte wave}
Scholte waves are boundary waves that propagate along elastic-acoustic interfaces. This problem is designed to the test numerical flux between acoustic and elastic media. In our problem setting, we consider two half-spaces: the upper half, $x_2>0$, is fluid with acoustic material parameters $\lambda_1$, $\mu_1=0$, and $\rho_1$. The lower half, $x_2<0$, is solid with elastic material parameters $\lambda_2$, $\mu_2$, and $\rho_2$.  The displacement of a Scholte wave in the acoustic region is given by
\begin{equation*}
\begin{split}
u_1&=\operatorname{Re}\left(i\kappa B_1e^{-\kappa b_{1p}x_2}e^{i\left(\kappa x_1-\omega t\right)}\right),\\ 
u_2&=\operatorname{Re}\left(-\kappa b_{1p}B_1e^{-\kappa b_{1p}x_2}e^{i\left(\kappa x_1-\omega t\right)}\right),
\end{split}
\end{equation*}
and in the elastic region by
\begin{equation*}
\begin{split}
u_1&=\operatorname{Re}\left(\left(i\kappa B_2e^{\kappa b_{2p}x_2}-\kappa b_{2s}B_3e^{\kappa b_{2s}x_2}\right)e^{i\left(\kappa x_1-\omega t\right)}\right),\\ 
u_2&=\operatorname{Re}\left(\left(\kappa b_{2p}B_2e^{\kappa b_{2p}x_2}+ikB_3e^{\kappa b_{2s}x_3}\right)e^{i\left(\kappa x_1-\omega t\right)}\right).
\end{split}
\end{equation*}
The wavenumber is $\kappa=\frac{\omega}{c}$, with decay rates
$$b_{1p}=\left(1-\frac{c^2}{c_{1p}^2}\right)^{\frac{1}{2}}, \qquad b_{2p}=\left(1-\frac{c^2}{c_{2p}^2}\right)^{\frac{1}{2}},\qquad b_{2s}=\left(1-\frac{c^2}{c_{2s}^2}\right)^{\frac{1}{2}},$$
where $c$ is the Scholte wavespeed. The longitudinal and transverse wavespeeds are
$$c_{1p}=\sqrt{\frac{\lambda_1+2\mu_1}{\rho_1}},\qquad c_{2p}=\sqrt{\frac{\lambda_2+2\mu_2}{\rho_2}},\qquad c_{2s}=\sqrt{\frac{\mu_2}{\rho_2}}.$$
The wave amplitudes are related to each other through the interface condition (\ref{eq:condition})
\begin{equation}
\begin{split}
2i\left(1-\frac{c^2}{c^2_{2p}}\right)^{\frac{1}{2}}B_2-\left(2-\frac{c^2}{c^2_{2s}}\right)B_3&=0,\\
\frac{c^2}{c^2_{2s}}B_1+\frac{\rho_2}{\rho_1}\left(2-\frac{c^2}{c^2_{2s}}\right)B_2+2i\frac{\rho_2}{\rho_1}\left(1-\frac{c^2}{c^2_{2s}}\right)^{\frac{1}{2}}B_3&=0,\\
\left(1-\frac{c^2}{c^2_{1p}}\right)^{\frac{1}{2}}B_1+\left(1-\frac{c^2}{c^2_{2p}}\right)^{\frac{1}{2}}B_2+iB_3&=0.
\end{split}
\label{eq:amplitude}
\end{equation}
The Scholte wavespeed $c$ is chosen such that the determinant of (\ref{eq:amplitude}) is zero, and $c$ satisfies
$$\left(\frac{\rho_1}{\rho_2}b_{2p}+b_{1p}\right)r^4-4b_{1p}r^2-4b_{1p}\left(b_{2p}b_{2s}-1\right)=0,$$
where $r=c/c_{2s}$.

We choose the acoustic and elastic material parameters as $\lambda_1=1,\ \rho_1=1,\ \mu_1=0$, and $\lambda_2=\mu_2=1,\ \rho_2=1$. For these material parameters, we obtain $c=0.7110017230197$ and choose $B_1=-i0.3594499773037$, $B_2=-i0.8194642725978$, and $B_3=1$. In our experiment, we choose a uniform mesh with different size $h$ covering a square domain $\left[-1,1\right]^2$. As with Snell's law, we investigate the convergence rates of the proposed method for a central flux ($\tau_u=\tau_p=0$) and a penalty flux ($\tau_u=\tau_p=1$).
\begin{figure}
	\centering
	\setcounter{subfigure}{0}
	\subfloat[Central flux ($\tau_u=\tau_p=0$)]{
		\begin{tikzpicture}
		\begin{loglogaxis}[
		width=0.5\textwidth,
		title style = {font=\large},
		legend style = {font=\tiny},
		xlabel=Mesh size $h$,
		ylabel=$L^2$ Error,
		xlabel style= {font=\scriptsize},
		ylabel style= {font=\scriptsize},
		legend pos = south east
		]
		\addplot[color=blue,mark=diamond] coordinates {
			(0.25,0.351501)
			(0.125, 0.185117)
			(0.0625,0.0959123)
			(0.0312,0.0491434)
		};
		
		\addplot[color=red,mark=o] coordinates {
			(0.25,0.0335087)
			(0.125, 0.00669997)
			(0.0625,0.00144653)
			(0.0312,0.000333904)
		};

		\addplot[color=blue,mark=o] coordinates {
			(0.25,0.00386704)
			(0.125, 0.000526297)
			(0.0625,0.0000683133)
			(0.03125,0.00000868354)
		};
		\addplot[color=red,mark=triangle] coordinates {
			(0.25,0.000212605)
			(0.125, 0.0000117583)
			(0.0625,0.000000659653)
			(0.0312,0.0000000384745)
		};
		\logLogSlopeTriangle{0.2}{0.1}{0.805}{0.96}{blue};
		\logLogSlopeTriangle{0.2}{0.1}{0.55}{2.12}{red};
		\logLogSlopeTriangle{0.2}{0.1}{0.36}{2.98}{blue};
		\logLogSlopeTriangle{0.2}{0.1}{0.08}{4.10}{red};
		\legend{$N=1$,$N=2$,$N=3$,$N=4$}
		\end{loglogaxis}
		
		\end{tikzpicture}}
		\subfloat[Penalty flux ($\tau_u=\tau_p=1$)]{
\begin{tikzpicture}
\begin{loglogaxis}[
width=0.5\textwidth,
title style = {font=\large},
legend style = {font=\tiny},
xlabel=Mesh size $h$,
xlabel style= {font=\scriptsize},
ylabel style= {font=\scriptsize},
legend pos = south east
]

\addplot[color=blue,mark=diamond] coordinates {
	(0.25,0.124395)
	(0.125, 0.0359558)
	(0.0625,0.00968633)
	(0.0312,0.00259978)
};

\addplot[color=red,mark=o] coordinates {
	(0.25,0.0159341)
	(0.125, 0.00190554)
	(0.0625,0.000233487)
	(0.0312,0.0000292817)
};

\addplot[color=blue,mark=o] coordinates {
	(0.25,0.00094556)
	(0.125, 0.0000635446)
	(0.0625,0.0000040409)
	(0.03125,0.000000253035)
};
\addplot[color=red,mark=triangle] coordinates {
	(0.25,0.0000344301)
	(0.125, 0.00000113083 )
	(0.0625,0.0000000360469)
	(0.0312,0.00000000122477)
};
\logLogSlopeTriangle{0.2}{0.1}{0.735}{1.90}{blue};
\logLogSlopeTriangle{0.2}{0.1}{0.535}{3.00}{red};
\logLogSlopeTriangle{0.2}{0.1}{0.32}{4.00}{blue};
\logLogSlopeTriangle{0.2}{0.1}{0.08}{4.88}{red};
\legend{$N=1$,$N=2$,$N=3$,$N=4$}
\end{loglogaxis}

\end{tikzpicture}}

\caption{Convergence of $L^2$ errors for the Scholte wave solution.}
\label{fig:Scholte}
\end{figure}
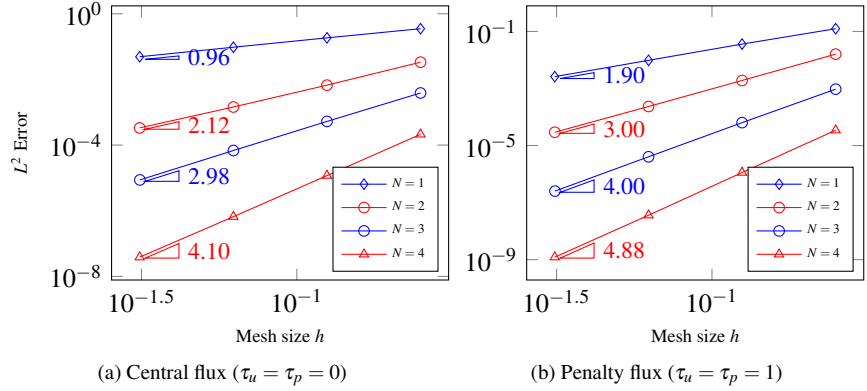

Figures~\ref{fig:snell} and ~\ref{fig:Scholte} show $L^2$ error for the Snell's law and Scholte waves at time $T=5$, respectively. For penalty fluxes, the computed convergence rate is close to the optimal rate of $O(h^{N+1})$. For central fluxes, we observe again an odd-even pattern, though the rate of convergence is one order lower than observed for Snell's law.  

\note{We also computed Scholte wave solutions using more realistic material coefficients from \cite{ye2016discontinuous}. The fluid media is homogeneous isotropic with an acoustic wavespeed of 1.5 km/s and density 1.0 g/cm\textsuperscript{3}. The solid media is homogeneous and isotropic with a P-wave speed of 3.0 km/s and an S-wave speed of 1.5 km/s, with a density of 2.5 g/cm\textsuperscript{3}. Errors for a Scholte wave solution at time $T=1$ are shown in Figure~\ref{fig:Scholte2}.}
	
	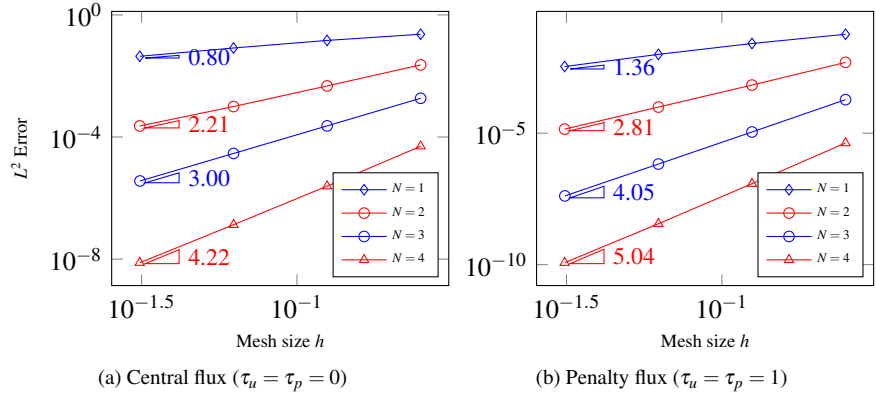
\begin{figure}[h]
		\centering
		\setcounter{subfigure}{0}
		\subfloat[Central flux ($\tau_u=\tau_p=0$)]{
			\begin{tikzpicture}
			\begin{loglogaxis}[
			width=0.5\textwidth,
			title style = {font=\large},
			legend style = {font=\tiny},
			xlabel=Mesh size $h$,
			ylabel=$L^2$ Error,
			xlabel style= {font=\scriptsize},
			ylabel style= {font=\scriptsize},
			legend pos = south east
			]
			\addplot[color=blue,mark=diamond] coordinates {
				(0.25,0.232836)
				(0.125, 0.146085)
				(0.0625,0.0826024)
				(0.0312,0.0441329)
			};
			
			\addplot[color=red,mark=o] coordinates {
				(0.25,0.0230903)
				(0.125, 0.00471466)
				(0.0625,0.000997027)
				(0.0312,0.000230043)
			};

			\addplot[color=blue,mark=o] coordinates {
				(0.25,0.00185812)
				(0.125, 0.000231572)
				(0.0625,2.89548e-05)
				(0.03125,3.62749e-06)
			};
			\addplot[color=red,mark=triangle] coordinates {
				(0.25,4.98633e-05)
				(0.125, 2.45675e-06)
				(0.0625,1.33825e-07)
				(0.0312,7.62032e-09)
			};
			\logLogSlopeTriangle{0.2}{0.1}{0.83}{0.80}{blue};
			\logLogSlopeTriangle{0.2}{0.1}{0.575}{2.21}{red};
			\logLogSlopeTriangle{0.2}{0.1}{0.375}{3.00}{blue};
			\logLogSlopeTriangle{0.2}{0.1}{0.08}{4.22}{red};
			\legend{$N=1$,$N=2$,$N=3$,$N=4$}
			\end{loglogaxis}
			
			\end{tikzpicture}}
		\subfloat[Penalty flux ($\tau_u=\tau_p=1$)]{
			\begin{tikzpicture}
			\begin{loglogaxis}[
			width=0.5\textwidth,
			title style = {font=\large},
			legend style = {font=\tiny},
			xlabel=Mesh size $h$,
			xlabel style= {font=\scriptsize},
			ylabel style= {font=\scriptsize},
			legend pos = south east
			]
			
			\addplot[color=blue,mark=diamond] coordinates {
				(0.25,0.0577376)
				(0.125, 0.0256728)
				(0.0625,0.00991271)
				(0.0312,0.00342262)
			};
			
			\addplot[color=red,mark=o] coordinates {
				(0.25,0.00490543)
				(0.125, 0.000673714)
				(0.0625,9.83025e-05)
				(0.0312,1.42311e-05)
			};

			\addplot[color=blue,mark=o] coordinates {
				(0.25,0.000187929)
				(0.125,1.10885e-05)
				(0.0625,6.69427e-07)
				(0.03125,4.11108e-08)
			};
			\addplot[color=red,mark=triangle] coordinates {
				(0.25,4.2186e-06)
				(0.125, 1.19909e-07)
				(0.0625,3.59448e-09)
				(0.0312,1.18657e-10)
			};
			\logLogSlopeTriangle{0.2}{0.1}{0.79}{1.36}{blue};
			\logLogSlopeTriangle{0.2}{0.1}{0.565}{2.81}{red};
			\logLogSlopeTriangle{0.2}{0.1}{0.32}{4.05}{blue};
			\logLogSlopeTriangle{0.2}{0.1}{0.08}{5.04}{red};
			\legend{$N=1$,$N=2$,$N=3$,$N=4$}
			\end{loglogaxis}
			
			\end{tikzpicture}}
		
		\caption{Convergence of $L^2$ errors for the Scholte wave solution using material coefficients in \cite{ye2016discontinuous}.}
		\label{fig:Scholte2}
	\end{figure}
%
\subsection{Curvilinear meshes}\label{sec:curved}
We now present numerical experiments verifying the stability and accuracy of the DG scheme presented in \note{Section~\ref{sec:eawave}} for curvilinear meshes. We use isoparametric mappings in the following experiments, where the mapping from the reference element to each physical element is a polynomial of degree $N$. We start from a uniform triangular mesh on the square domain $\Omega=[-1,1]^2$ and place high-order Warp and Blend interpolation nodes on each element. The physical locations $(x_i,y_i)$ of these nodes are then perturbed to produce new nodal positions $(\tilde{x}_i,\tilde{y}_i)$, where
$$\tilde{x}_i=x_i+\frac{1}{8}\cos\left(\frac{3\pi}{2}x\right)\sin\left(\pi y\right),\qquad\tilde{y}_i=y_i+\frac{1}{8}\sin\left(\pi x\right)\sin\left(\pi y\right).$$
These new positions $(\tilde{x}_i,\tilde{y}_i)$ now define a coordinate mapping from the reference element to a curved physical element, producing the warped mesh in Figure~\ref{fig:curved}. This mesh warping is constructed such that $x$ and $y$ deformations of each element are of roughly the same magnitude, while leaving the positions of nodes on the boundary unchanged. 
\begin{figure}
	\centering
	\setcounter{subfigure}{0}
	\subfloat[$\tau_u=\tau_p=0$]{\includegraphics[width=0.35\linewidth]{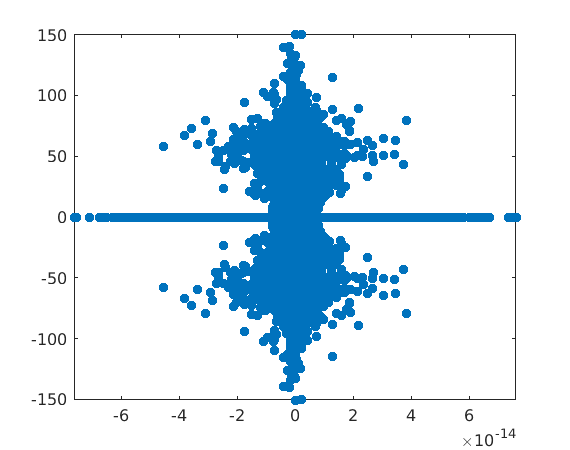}}
	\subfloat[$\tau_u=\tau_p=\frac{1}{2}$]{\includegraphics[width=0.35\linewidth]{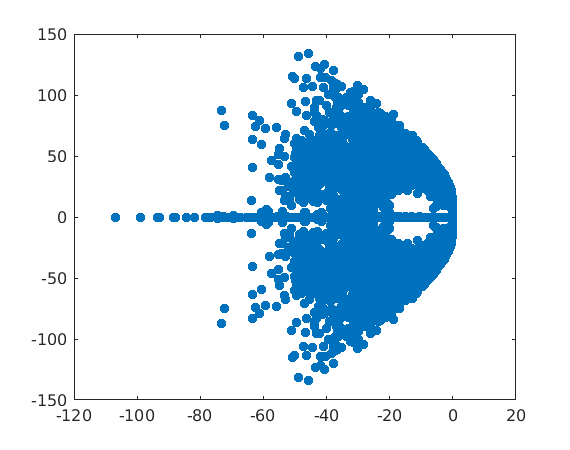}}
	\subfloat[$\tau_u=\tau_p=1$]{\includegraphics[width=0.35\linewidth]{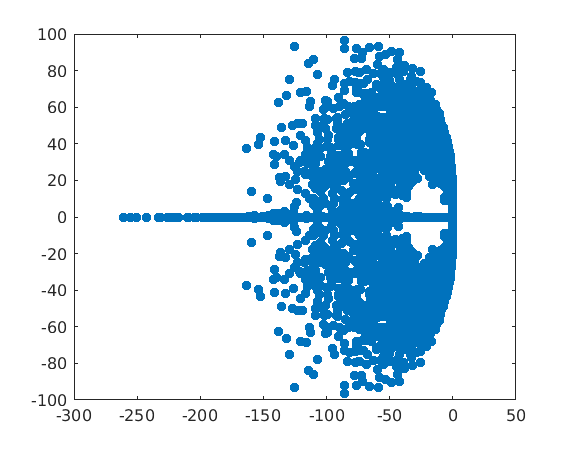}}\\
	\caption{Spectra of the discontinuous Galerkin discretization matrix for central and penalty fluxes on a warped curvilinear mesh of degree $N=3$.}
\label{fig:curvedspectra}
\end{figure}
\begin{figure}
	\setcounter{subfigure}{0}
	\subfloat[Central flux ($\tau_u=\tau_p=0$)]{
		\begin{tikzpicture}
		\begin{loglogaxis}[
		width=0.475\textwidth,
		title style = {font=\small},
		xlabel=Mesh size $h$,
		ylabel=$L^2$ Error,legend pos = south east,
		legend style = {font=\tiny}
		]
		\addplot[color=blue,mark=diamond] coordinates {
			(0.25,0.304203)
			(0.125, 0.167329)
			(0.0625,0.0883194)
			(0.0312,0.0456539)
		};
		\addplot[color=red,mark=o] coordinates {
			(0.25,0.0588248)
			(0.125, 0.0110107)
			(0.0625,0.00199906)
			(0.0312,0.000429199)
		};

		\addplot[color=blue,mark=o] coordinates {
			(0.25,0.0105581)
			(0.125, 0.00151951)
			(0.0625,0.000203999)
			(0.03125,0.0000259349)
		};
		\addplot[color=red,mark=triangle] coordinates {
			(0.25,0.00251472)
			(0.125, 0.0001611)
			(0.0625,0.00000813979)
			(0.0312,0.000000424105)
		};
		\logLogSlopeTriangle{0.2}{0.1}{0.79}{0.95}{blue};
		\logLogSlopeTriangle{0.2}{0.1}{0.505}{2.22}{red};
		\logLogSlopeTriangle{0.2}{0.1}{0.33}{2.98}{blue};
		\logLogSlopeTriangle{0.2}{0.1}{0.08}{4.26}{red};
		\legend{$N=1$,$N=2$,$N=3$,$N=4$}
		\end{loglogaxis}
		
		\end{tikzpicture}}
		\subfloat[Penalty flux ($\tau_u=\tau_p=1$)]{
		\begin{tikzpicture}
		\begin{loglogaxis}[
		width=0.475\textwidth,
		title style = {font=\small},
		xlabel=Mesh size $h$,
		ylabel=$L^2$ Error,legend pos = south east,
		legend style = {font=\tiny}
		]
		\addplot[color=blue,mark=diamond] coordinates {
			(0.25,0.121696)
			(0.125, 0.0391097)
			(0.0625,0.0114519)
			(0.0312,0.00322833)
		};
		
		\addplot[color=red,mark=o] coordinates {
			(0.25,0.0189815)
			(0.125, 0.002894)
			(0.0625,0.000404456)
			(0.0312,0.0000543826)
		};

		\addplot[color=blue,mark=o] coordinates {
			(0.25,0.00328983)
			(0.125, 0.000240836)
			(0.0625,0.0000161137)
			(0.03125,0.00000100707)
		};
		\addplot[color=red,mark=triangle] coordinates {
			(0.25,0.000532739)
			(0.125, 0.0000196195)
			(0.0625,0.000000693116)
			(0.0312,0.0000000226609)
		};
		\logLogSlopeTriangle{0.2}{0.1}{0.715}{1.83}{blue};
		\logLogSlopeTriangle{0.2}{0.1}{0.495}{2.89}{red};
		\logLogSlopeTriangle{0.2}{0.1}{0.285}{4.00}{blue};
		\logLogSlopeTriangle{0.2}{0.1}{0.08}{4.93}{red};
		\legend{$N=1$,$N=2$,$N=3$,$N=4$}
		\end{loglogaxis}
		
		\end{tikzpicture}}
	\caption{Convergence for the Scholte wave problem on curvilinear meshes.}
\label{fig:curvedconvergence}
\end{figure}
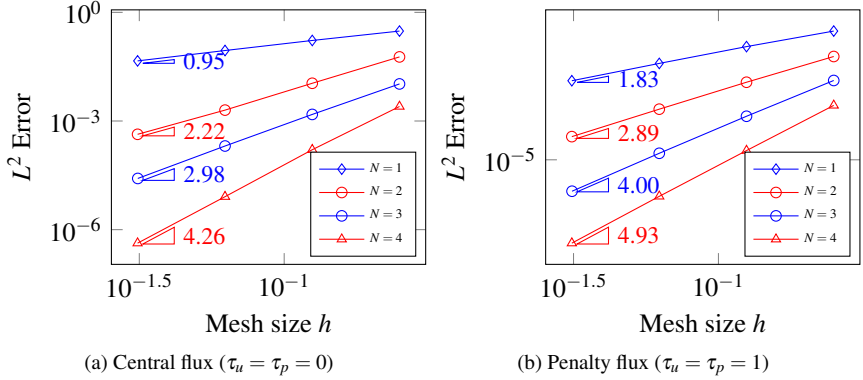

Figure~\ref{fig:curvedspectra} shows computed eigenvalues of the DG discretization matrix for $N=3$ and a warped curvilinear mesh. We use the strong-weak formulation introduced in Section~\ref{sec:curvedDG} and consider both central and penalty fluxes. We observe that for both central and penalty fluxes, the real part of all eigenvalues is non-positive (up to machine precision), verifying that the proposed DG scheme is energy stable. The introduction of the curvilinear warping appears to result in a magnification of the real and imaginary parts of larger magnitude eigenvalues, which also induces a smaller time-step size.  

We compute $L^2$ errors on a sequence of refined curvilinear meshes for $N=1,2,3,4$.  From Figure~\ref{fig:curvedconvergence}, we observe the rates of convergence of $L^2$ errors are consistent with the rates observed for affine meshes in Section~\ref{sec:classical}.

\section{Application examples}\label{sec:app}

In this section, we demonstrate the accuracy and flexibility of the proposed DG method for some application-based problems. In the first example, we simulate wave propagation through heterogeneous and anisotropic media. In the second example, we present an application of the new DG method to an inverse \note{problem} in photoacoustic tomography (PAT).

\subsection{Heterogeneous anisotropic media}
We examine a model wave propagation problem in heterogeneous and anisotropic media. In our experiments, we use two different experimental settings based on \cite{komatitsch2000simulation}. We divide the domain into three parts and set the left half (i.e. $x<0$) to be anisotropic elastic media, the right-bottom part (i.e. $x>0,y<0$) to be isotropic elastic media, and the right-upper part (i.e. $x>0,y>0$) to be acoustic media. We assume that the density $\rho=7100$ is constant over the whole domain.  

In the first experiment, we simulate wave propagation through homogeneous media. The entries of the stiffness matrix $\bm{C}$ in the anisotropic media are taken to be
$${C}_{11} = 0.165,\quad {C}_{12} = 0.05, \quad {C}_{22}=0.062,\quad {C}_{33}=0.0396,\qquad x<0,$$
$${C}_{11} = 0.165,\quad {C}_{12} = 0.0858, \quad {C}_{22}=0.165,\quad {C}_{33}=0.0396,\qquad x>0,\ y<0,$$
and the acoustic wavespeed is set to be
$$c =\sqrt{\frac{C_{11}}{\rho}}, \qquad x>0,\ y>0.$$

In the second experiment, we introduce sub-cell \note{heterogeneities} to the material parameters.  For the isotropic elastic region $x<0,y>0$, we set 
\begin{equation*}
\begin{split}
{C}_{11} &= 0.165\left(1+\frac{1}{4}\sin\left(\frac{x}{0.08}\pi\right)\right),\quad {C}_{12} = 0.05,\\ {C}_{22}&=0.062\left(1+\frac{1}{4}\sin\left(\frac{x}{0.08}\pi\right)\right),\quad {C}_{33}=0.0396\left(1+\frac{1}{4}\sin\left(\frac{x}{0.08}\pi\right)\right),
\end{split}
\end{equation*}
and for the anisotropic elastic region $x<0,y<0$
\begin{equation*}
\begin{split}
{C}_{11} &= 0.165\left(1+\frac{1}{4}\sin\left(\frac{x}{0.08}\pi\right)\right),\quad {C}_{12} = 0.0858,\\ {C}_{22}&=0.165\left(1+\frac{1}{4}\sin\left(\frac{x}{0.08}\pi\right)\right),\quad {C}_{33}=0.0396\left(1+\frac{1}{4}\sin\left(\frac{x}{0.08}\pi\right)\right).
\end{split}
\end{equation*}
In the acoustic domain $x>0,y>0$, we again set
$$c =\sqrt{\frac{C_{11}}{\rho}}.$$

In all experiments, we set the order of approximation $N=5$.  We use a uniform triangular mesh of 32768 elements on domain $[-0.32,0.32]^2$.  Forcing is applied to the $y$ component of the velocity using a Ricker wavelet point source
$$f(\bm{x},t)=\left(1-2\left(\pi f_0\left(t-t_0\right)\right)^2\right)e^{-\left(\pi f_0\left(t-t_0\right)\right)^2}\delta\left(x-x_0\right),$$
where $x_0=-0.02$, $f_0=0.17$, and $t_0=1/f_0$. 
\begin{figure}
	\centering
	\setcounter{subfigure}{0}
	\subfloat[$T=30$]{\includegraphics[width=0.44\linewidth]{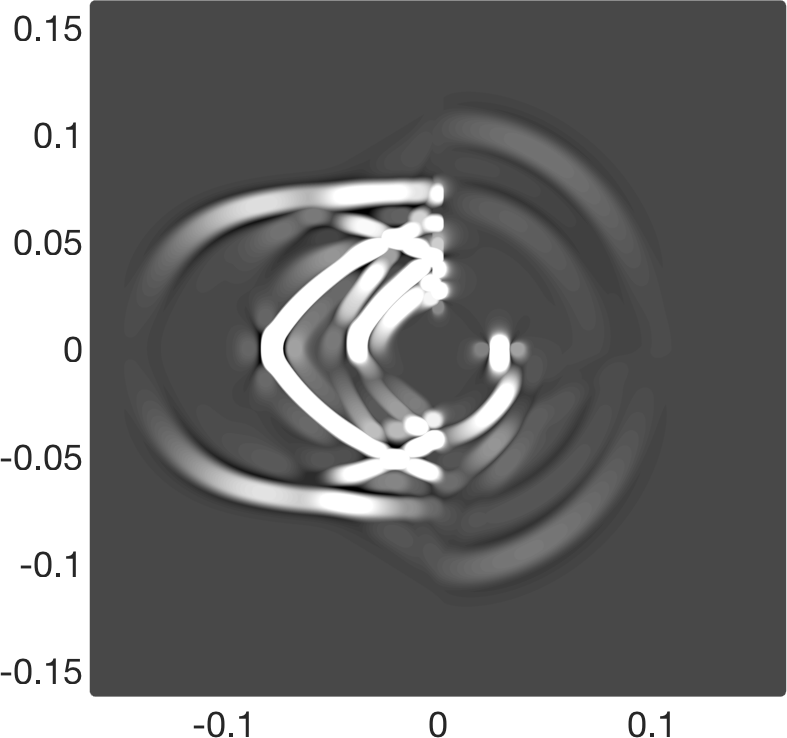}}\hspace{0.2cm}
	\subfloat[$T=60$]{\includegraphics[width=0.5\linewidth]{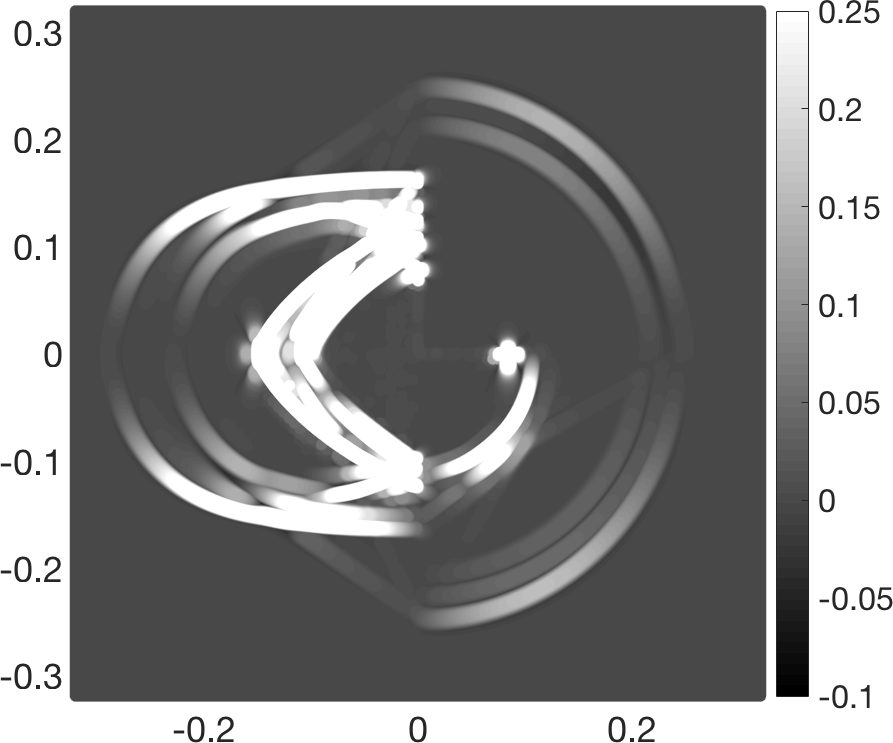}}
	\caption{An example of wave propagation in homogeneous anisotropic-isotropic acoustic-elastic media.}
\label{fig:homo}
\end{figure}
\begin{figure}
	\centering
	\subfloat[$T=30$]{\includegraphics[width=0.44\linewidth]{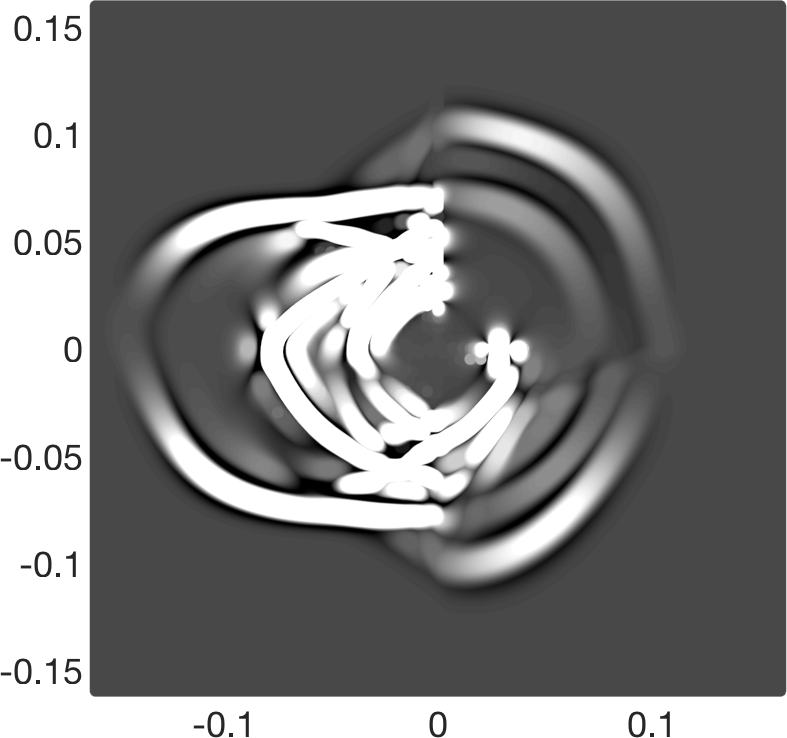}}\hspace{0.2cm}
	\subfloat[$T=60$]{\includegraphics[width=0.5\linewidth]{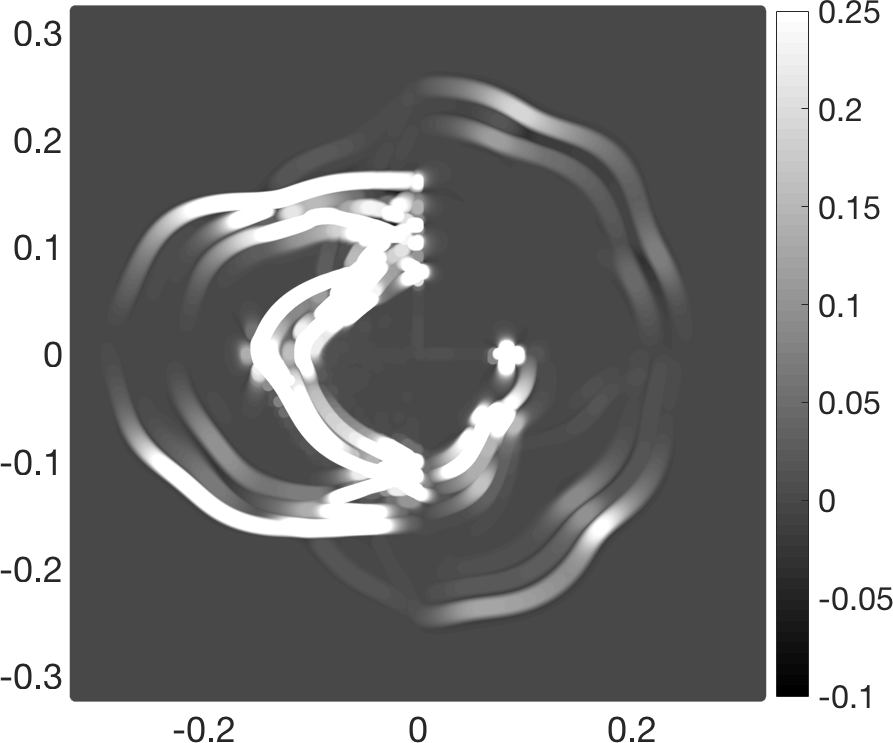}}
	\caption{An example of wave propagation in heterogeneous anisotropic-isotropic acoustic-elastic media.}
\label{fig:inhomo}
\end{figure}

In all implementations, we take the penalty parameters to be $\tau_u=\tau_p=1/2$.  \note{For this value of $\tau$ and for the acoustic wave equation in homogeneous media, the penalty flux coincides with the upwind flux. Moreover, numerical results suggest that the maximum stable time-step size for $\tau = 1/2$ is the same as the maximum stable time-step size for $\tau = 0$ \cite{chan2018weight}, which suggests that this level of dissipation does not require a more restrictive CFL condition}.  Figures~\ref{fig:homo} and \ref{fig:inhomo} show the $y$ component of velocity at times $T=30\mu s$ and $T=60 \mu s$.  In the elastic regions, the results agree with the reference results in \cite{komatitsch2000simulation}.  In the elastic-acoustic regions, we observe the presence of a propagating pressure wave, while the stress wave ends in a Scholte-type wave propagating along the acoustic-elastic interface.  Figure~\ref{fig:inhomo} illustrates the effect of media heterogeneities, which manifest as a spatially-dependent warping of the solution.  

\subsection{Photoacoustic tomography}\label{sec:PAT}

Photoacoustic tomography (PAT) is an imaging modality which takes advantage of high-contrast exhibited by optical absorption and the high resolution available for broadband acoustic waves in soft biological tissues.  PAT relies on the so-called ``photoacoustic effect'':  a short microwave or light pulse is sent through a patient’s body which slightly heats up tissue.  The expansion due to heat generates weak acoustic waves, which are measured away from the patient’s body.  The main step of PAT is the recovery of the initial acoustic profile, which in turn provides information about the rate of absorption and tissue properties at different points in the body.  

Given the initial state of the pressure field $P$, the forward mapping $\mathcal{F}$ propagates the wave field to the  measurements $M$ (Dirichlet data)  on the boundary $(0,T) \times \partial \Omega$. In practice, to produce synthetic measurements, an absorbing boundary condition is employed to allow the waves to radiate outwardly without spurious reflections. The goal of PAT is to invert the forward mapping $\mathcal{F} : P \mapsto M$. Typically, a time-reversal method is utilized to approximately invert this forward mapping. The time-reversal mapping $\mathcal{R}$ consists of running the wave system backwards in time, from vanishing final condition at $\{t = T\} \times \Omega$, driven from the boundary $(0,T) \times \partial \Omega$ by the time-reversed boundary measurements $M$ as Dirichlet data.  The resulting pressure profile at $\{ t=0 \} \times \Omega$ is an approximation of the original profile $P$.

This approach can be inaccurate for short times and heterogeneous media. However, the quality of the reconstruction can be improved by approximating the exact inversion operator using a truncated Neumann series \cite{qian2011efficient}.  Similar reconstruction algorithms have been introduced for several variations of the wave equation \cite{stefanov2009thermoacoustic,stefanov2015multiwave,acosta2015multiwave, homan2012multi,palacios2016reconstruction}. We follow the approach proposed in \cite{acosta2018thermoacoustic}, which is summarized in Algorithm~\ref{alg:pat}. These approaches rely on the following error estimate,
\begin{equation*}
\| \text{Id} - \mathcal{R} \mathcal{F} \|_{L^{2}(\Omega)} \leq \kappa < 1,
\end{equation*}
which is verifiable when the wave speed is non--trapping (see details in \cite{stefanov2009thermoacoustic}). In other words, the time-reversal mapping $\mathcal{R}$ inverts the forward operator $\mathcal{F}$ up to a contraction mapping. Algorithm~\ref{alg:pat} is then the application of a fixed point iteration or truncated Neumann series. The error associated with the $n^{\text{th}}$ iteration satisfies,
\begin{equation*}
\| P - P_{n} \|_{L^{2}(\Omega)} \leq \| P_{0} \|_{L^{2}(\Omega)} \frac{\kappa^{n+1}}{1 - \kappa},
\end{equation*}
where $\kappa < 1$.

\begin{algorithm}
	\caption{Time-reversal algorithm for PAT}\label{euclid}
	\begin{algorithmic}[1]
		\Procedure{Initial time-reversal given boundary measurements}{}
		\State Solve the wave propagation problem backwards in time with boundary conditions driven by boundary measurements and zero final time condition. 
		\State Store the pressure field at time $t=0$ in $P_{0}$.
		\EndProcedure
		\Procedure{Forward and Backward Iteration}{}
		\For{n=1:Max iteration} 		
		\State Apply the forward solver with initial condition $P_{n-1}$ and absorbing boundary conditions.  Store the solution at time $t=T$ in $P_f$.
		\State Apply the backward solver with initial condition $P_f$ and zero Dirichlet boundary condition.  Store the solution at time $t=0$ in $P_b$.
		\State Update $P_{n} = P_{n-1} + P_b$.
	    \EndFor			
		\EndProcedure
	\end{algorithmic}
	\label{alg:pat}
\end{algorithm}

We test our PAT algorithm by reconstructing portions of the Shepp-Logan phantom (SLP) , which is a standard test for image reconstruction algorithms. The SLP is defined as the sum of 10 ellipses inside the computational domain $[-1,1]^2$.  The specific setting of our experiment is presented in Table~\ref{tab:slp}, and we arbitrarily set the penalty parameters to be $\tau_u = \tau_p = \tau_\sigma = \tau_v = 1$. We simply use the even polynomial function in \cite{yu2005differentiable} to construct a smoothed Shepp-Logan phantom for our numerical simulations with smoothing parameters $m=2,n=4$.  

We modify the typical SLP to emulate physical settings found for a human skull.  We consider the domain inside domain of Ellipse a and outside of Ellipse b as skull modeled by elastic media.  The rest of the domain is acoustic. The meshes (see in Figure~\ref{fig:refinedmesh} and~\ref{fig:coarsemesh}) for the SLP is generated by \textbf{MESH2D} \cite{engwirda2009mesh2d}, a MATLAB-based mesh-generator for two-dimensional geometries.  We use two meshes to test our PAT solver and compare results. The fine mesh consists of 7626 nodes and 14994 elements.  The thinnest portion of the elastic domain is resolved using three layers of elements. The coarse mesh consists of 4190 nodes and 8122 elements, and the thinnest portion of the elastic strip is resolved \note{using} only one or two layers of elements.  

We generate synthetic boundary data by running a forward problem and saving boundary measurements up to final time \textcolor{black}{$T=2$}.  We implement two versions of PAT: the first uses forward and backward solvers based on the discussed acoustic-elastic DG formulation, while the second uses a purely acoustic solver for comparison.  The wavespeed for the purely acoustic solver is set to be the pressure wavespeed for the elastic system.  All experiments are run on an Nvidia TITAN GPU, and the solvers are implemented in the Open Concurrent Compute Abstraction framework (OCCA) \cite{medina2014occa} for clarity and portability.

\begin{table}
\centering
 \begin{tabular}{||c c c c c c||} 
 \hline
 Ellipse & Center & Major Axis & Minor Axis& Theta & Value \\ [0.5ex] 
 \hline\hline
 a & $(0,0)$ & 0.69& 0.92 & 0&0\\ 
 b & $(0,-0.0184)$ & 0.6624 & 0.874 &0 &0\\
 c & $(0.22,0)$ & 0.11 & 0.31& $-0.18^{\circ}$&0.02 \\
 d & $(-0.22,0)$ & 0.16 & 0.41& $0.18^{\circ}$&0.02\\
 e & $(0,0.35)$ & 0.21& 0.25&0&0.01 \\ 
 f & $(0,0.1)$ & 0.046 & 0.046 &0&0.01\\ 
 g & $(0,-0.1)$ & 0.046 & 0.046&0&0.01 \\ 
 h & $(-0.08,-0.605)$ & 0.046 & 0.023&0 &0.01\\ 
 i & $(0,-0.605)$ & 0.023 & 0.023 &0&0.01\\ 
 j & $(0.06,-0.605)$ & 0.023 & 0.046&0&0.01 \\ 
 [1ex] 
 \hline
  \end{tabular}
\caption{Setting of Shepp-Logan Phantom.}
\label{tab:slp}
\end{table}

\begin{table}[]
	\centering
	\begin{tabular}{||c c c c c||} 
		\hline
		Iteration & Fine & Fine (acous) & Coarse & Coarse (acous)\\ [0.5ex] 
		\hline\hline
		1 & 0.140530 & 0.147435 &0.140556&0.147103\\ 
		2 & 0.094658 &0.133881 &0.094811&0.133508\\
		3 & 0.075081 &0.130397&0.075347& 0.130010 \\
		4 & 0.065585 &0.129331 &0.065941&0.128939 \\
		5 & 0.060577 &0.128973 &0.060998&0.128577\\ 
		[0.5ex] 
		\hline
	\end{tabular}
	\caption{Relative $L^2$ errors at each iteration.}
	\label{tab:convergence}
\end{table}

\begin{figure}
	\centering
	\setcounter{subfigure}{0}
	\subfloat[Mesh]{\includegraphics[width=0.4\linewidth,height=0.4\linewidth]{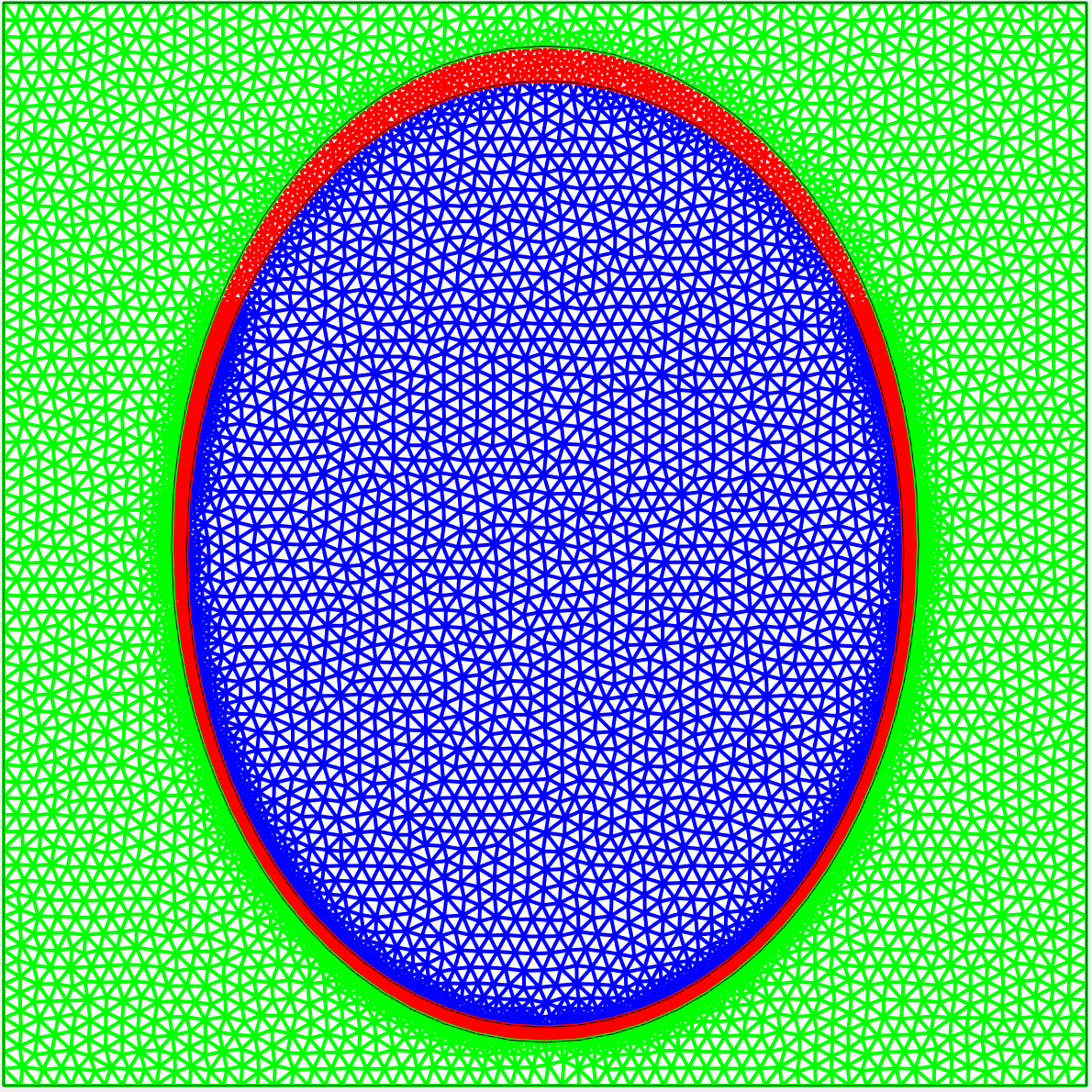}}
	\hskip 3ex
	\subfloat[Local mesh]{\includegraphics[width=0.4\linewidth,height=0.4\linewidth]{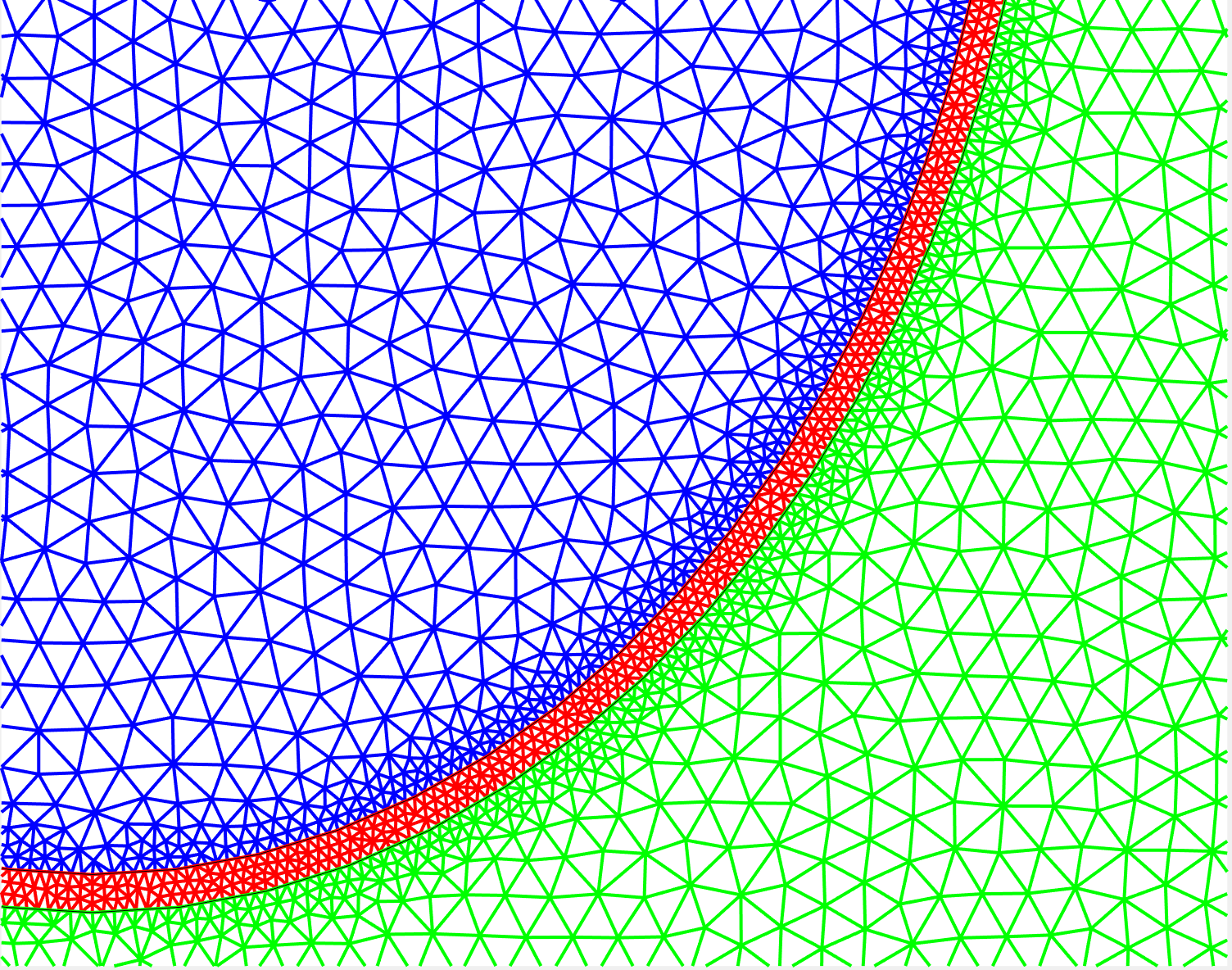}}
\caption{Fine mesh for the Shepp-Logan phantom.}
\label{fig:refinedmesh}
\end{figure}

\begin{figure}
	\centering
	\setcounter{subfigure}{0}
	\subfloat[Exact initial pressure]{\includegraphics[width=0.416\linewidth]{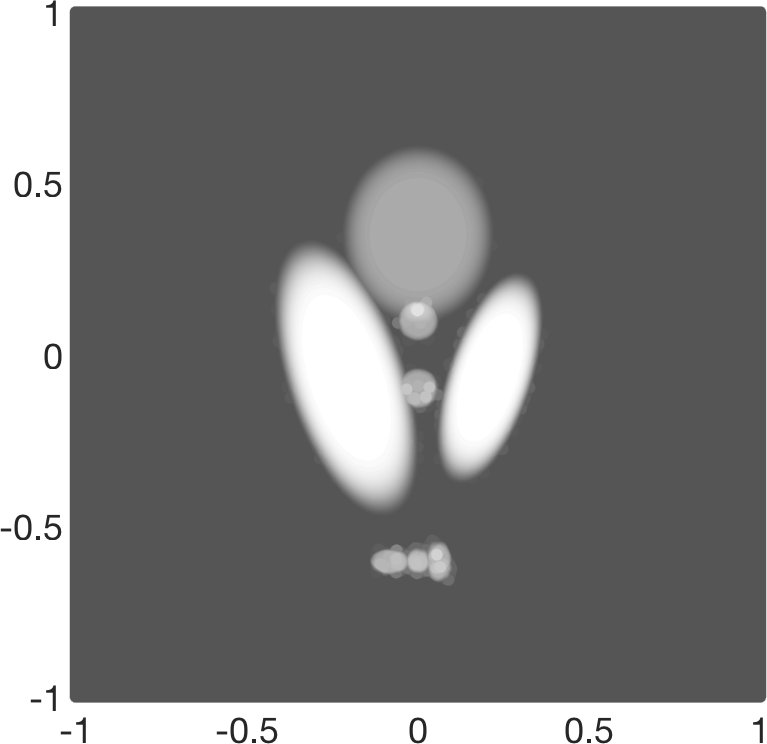}}\hspace{0.2cm}
	\subfloat[Purely acoustic reconstruction]{\includegraphics[width=0.5\linewidth]{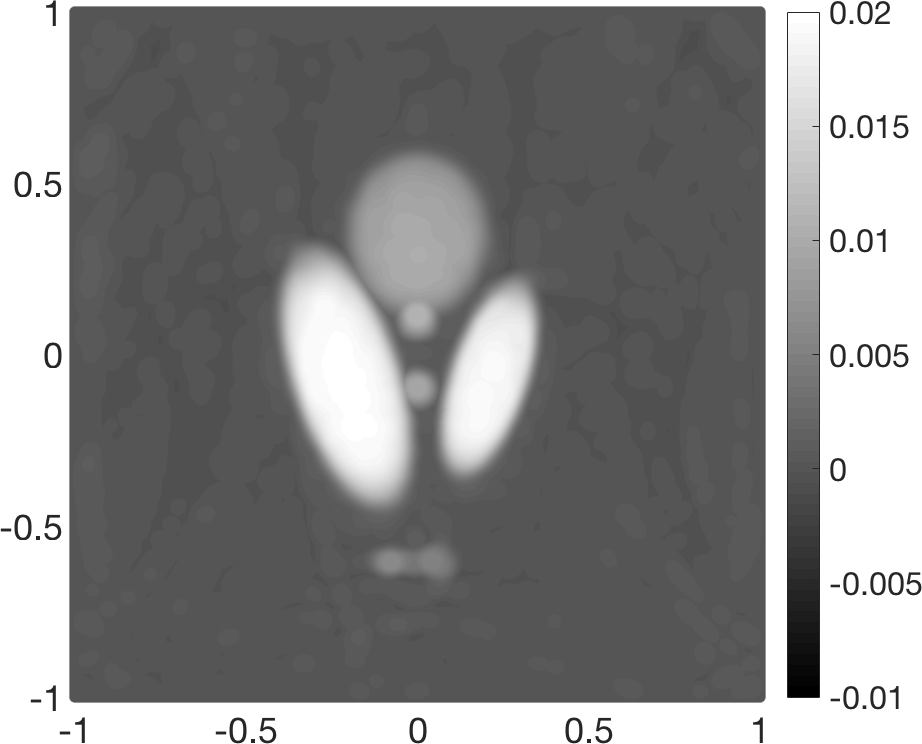}}\\
	\subfloat[Reconstruction after 1 iteration]{\includegraphics[width=0.416\linewidth]{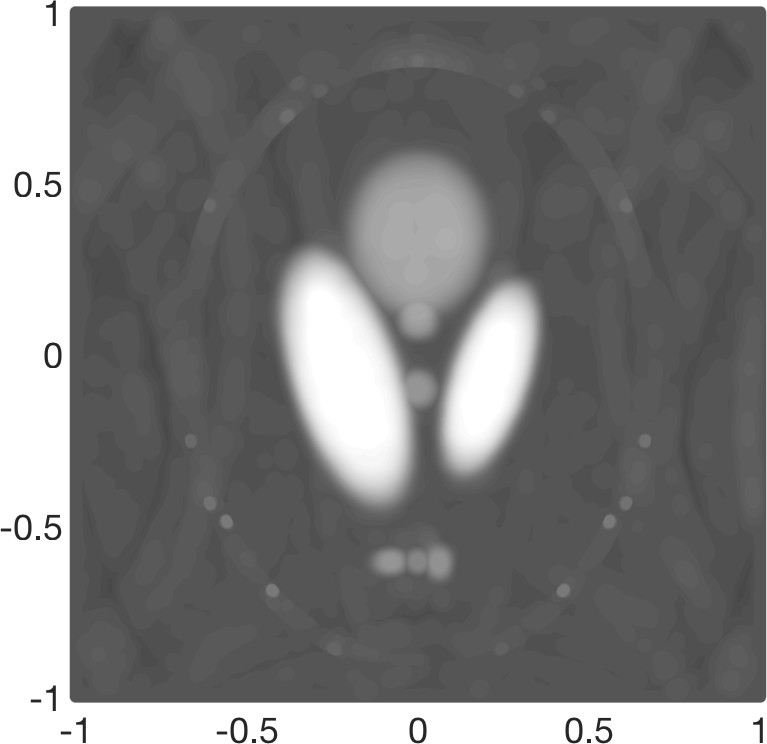}}\hspace{0.2cm}
	\subfloat[Reconstruction after 5 iterations]{\includegraphics[width=0.5\linewidth]{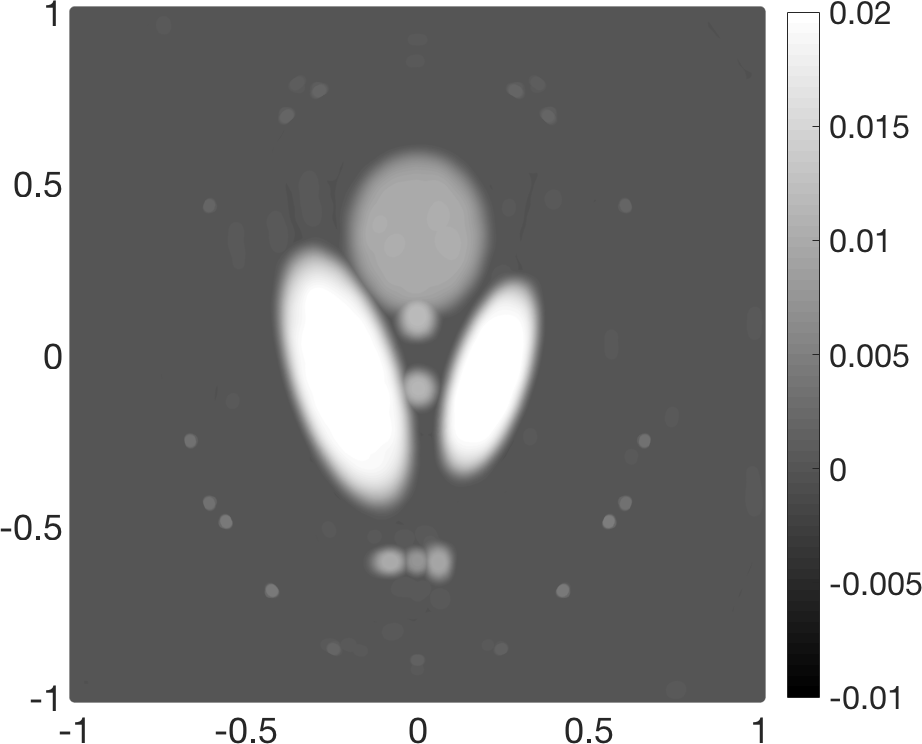}}
\caption{Reconstruction results using fine mesh.}
\label{fig:refinedPAT}
\end{figure}
\begin{figure}
	\centering
	\setcounter{subfigure}{0}
	\subfloat[Error after 1 iteration]{\includegraphics[width=0.416\linewidth]{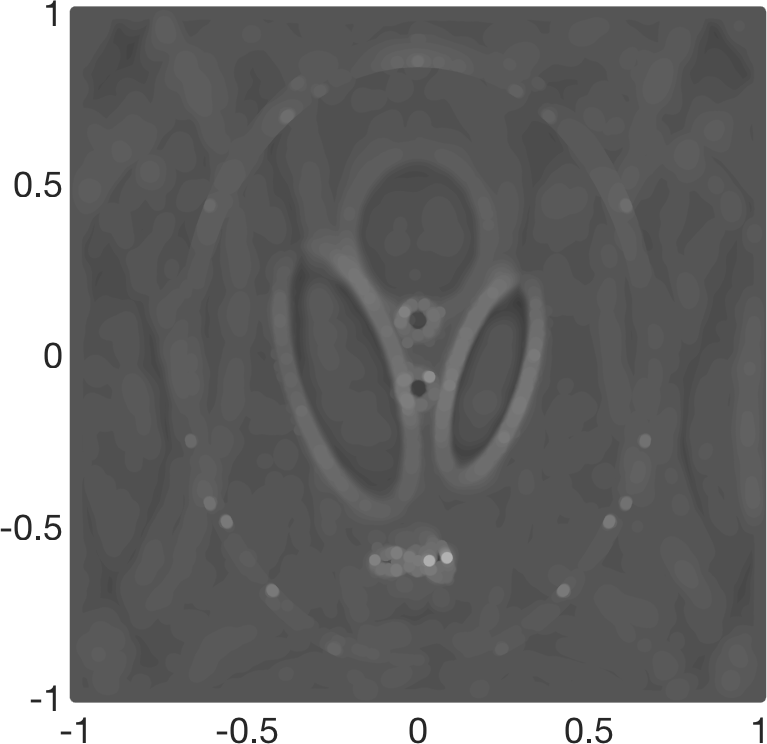}}\hspace{0.2cm}
	\subfloat[Error after 5 iteration]{\includegraphics[width=0.5\linewidth]{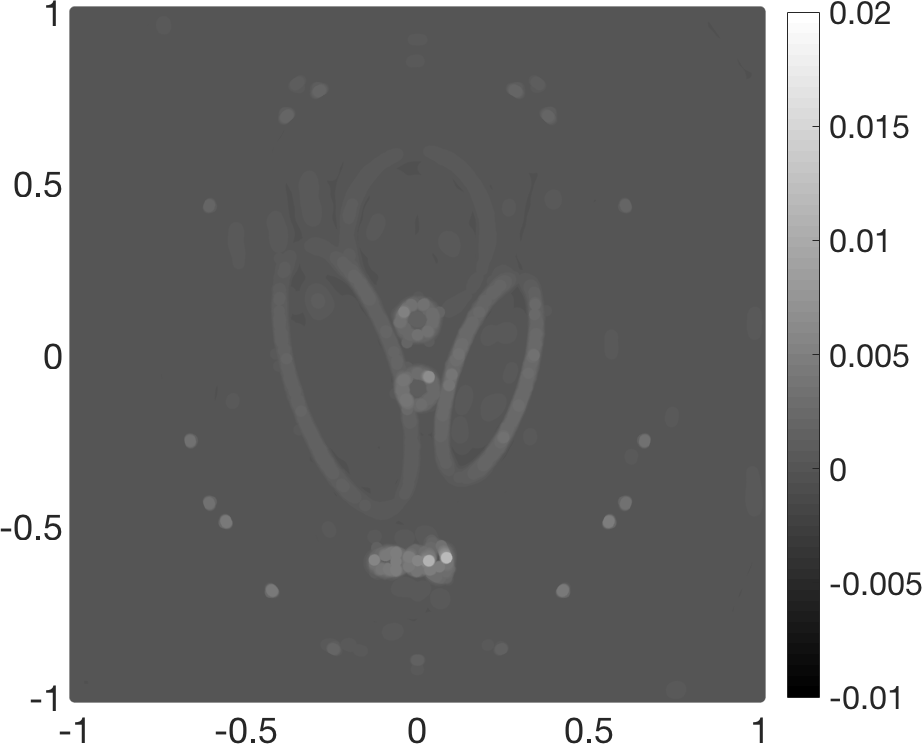}}
\caption{Reconstruction errors using fine mesh.}
\label{fig:refinederror}
\end{figure}

\begin{figure}
	\centering
	\setcounter{subfigure}{0}
	\subfloat[Mesh]{\includegraphics[width=0.4\linewidth,height=0.4\linewidth]{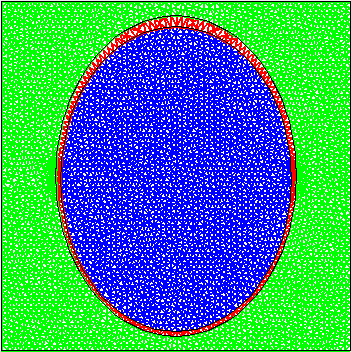}}
	\hskip 4ex
	\subfloat[Local mesh]{\includegraphics[width=0.4\linewidth,height=0.4\linewidth]{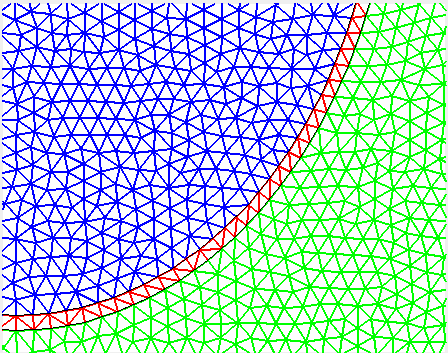}}
	\caption{Coarse mesh for the Shepp-Logan phantom.}
	\label{fig:coarsemesh}
\end{figure}

\begin{figure}
	\centering
	\setcounter{subfigure}{0}
	\subfloat[Exact initial pressure]{\includegraphics[width=0.416\linewidth]{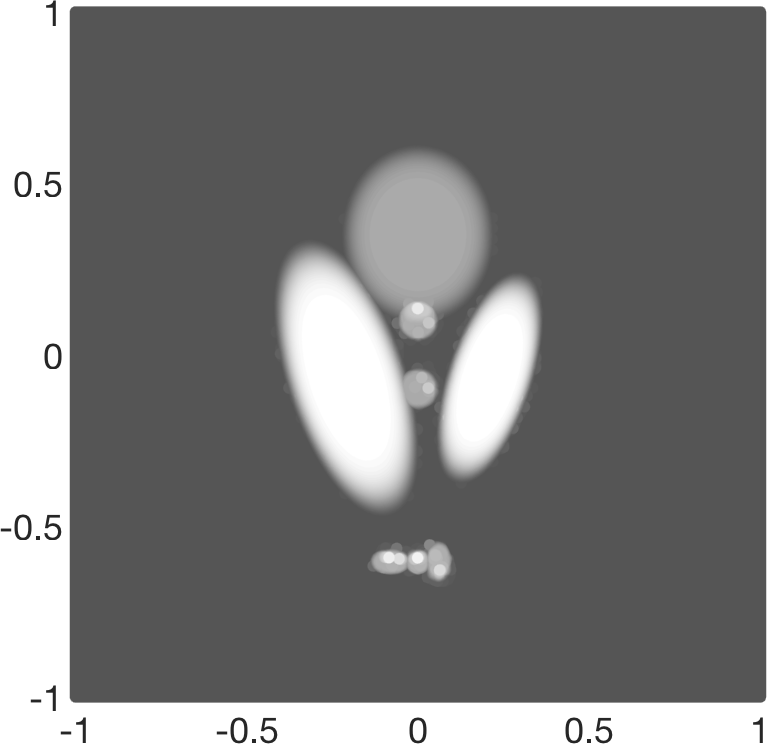}}\hspace{0.2cm}
	\subfloat[Purely acoustic reconstruction]{\includegraphics[width=0.5\linewidth]{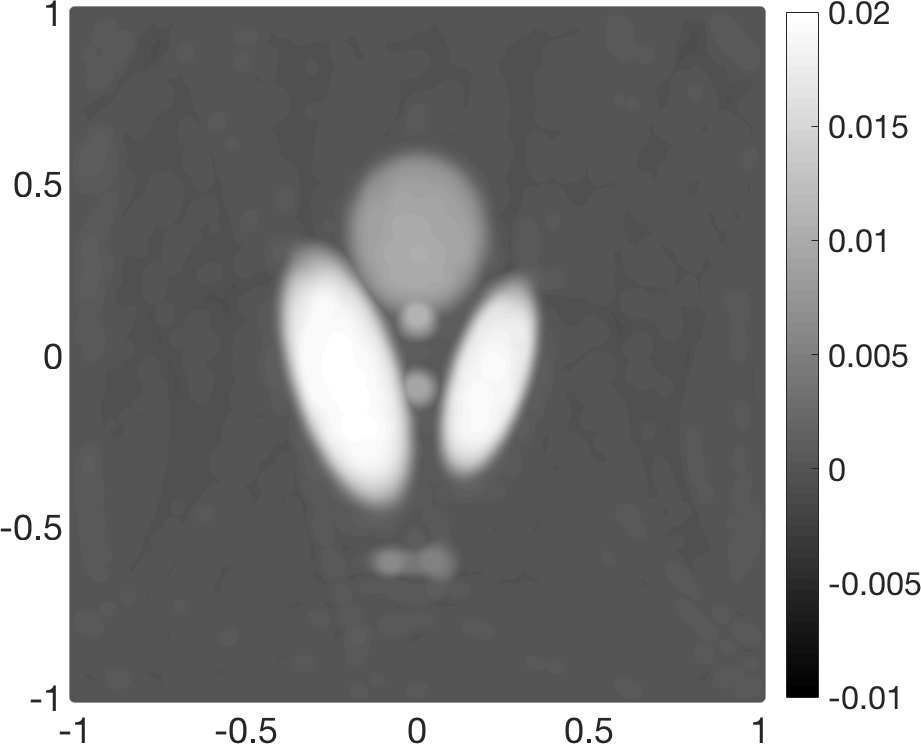}}\\
	\subfloat[Reconstruction after 1 iteration]{\includegraphics[width=0.416\linewidth]{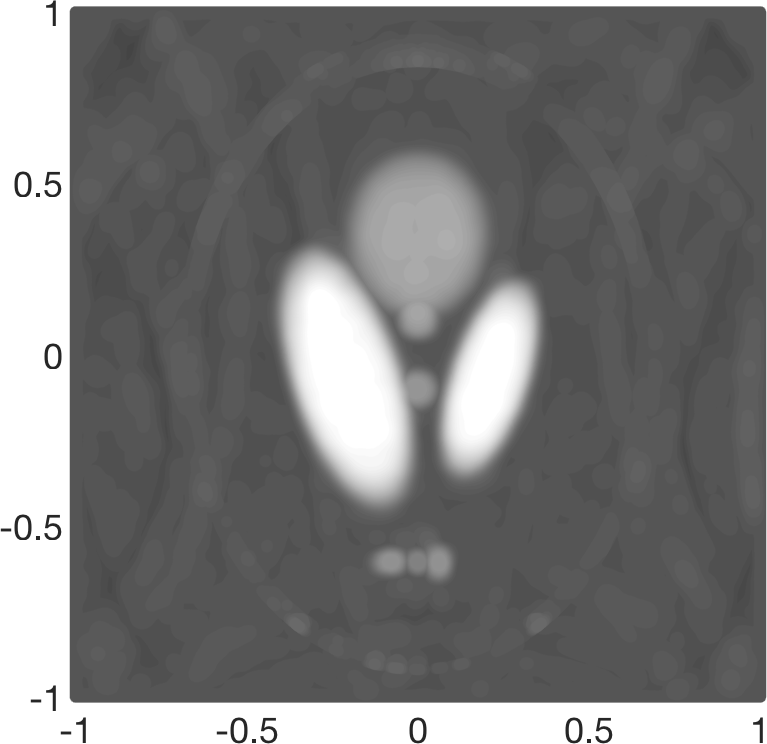}}\hspace{0.2cm}
	\subfloat[Reconstruction after 5 iterations]{\includegraphics[width=0.5\linewidth]{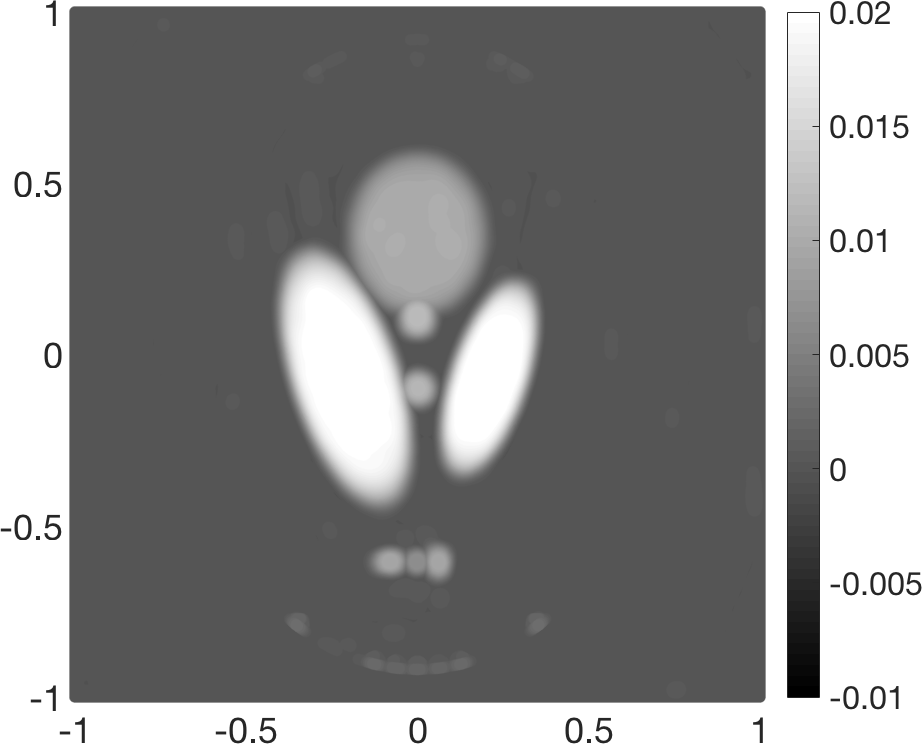}}
	\caption{Reconstruction results using coarse mesh.}
	\label{fig:coarsePAT}
\end{figure}
\begin{figure}
	\centering
	\setcounter{subfigure}{0}
	\subfloat[Error after 1 iteration]{\includegraphics[width=0.416\linewidth]{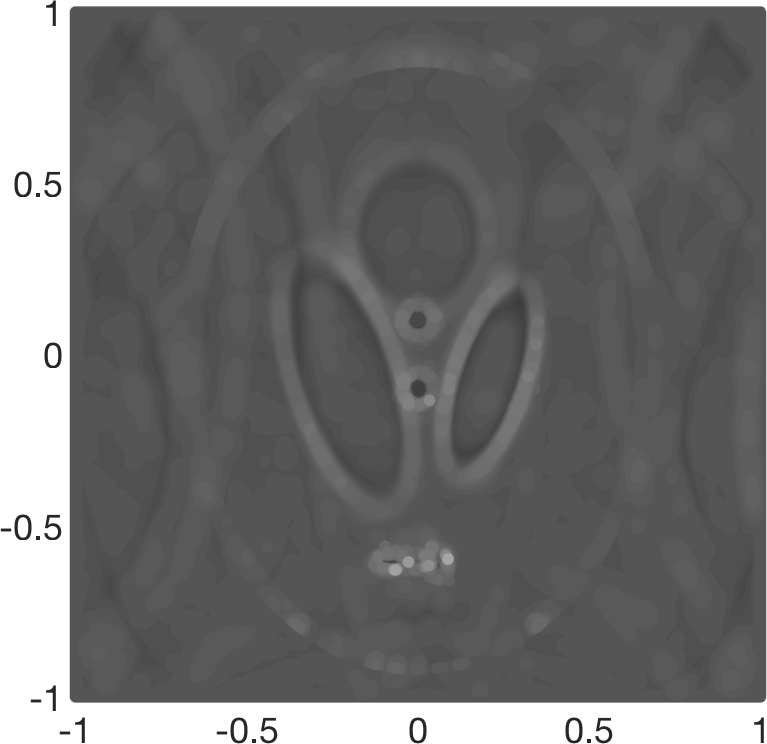}}\hspace{0.2cm}
	\subfloat[Error after 5 iteration]{\includegraphics[width=0.5\linewidth]{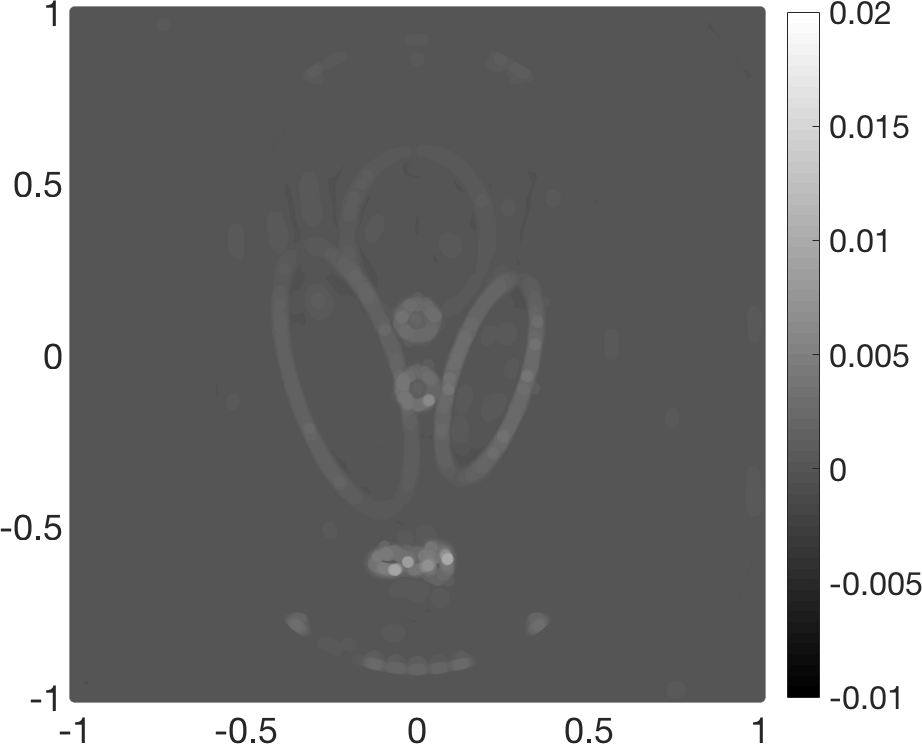}}
	\caption{Reconstruction errors using coarse mesh.}
	\label{fig:coarseerror}
\end{figure}
The relative $L^2$ errors during each iteration are presented in Table~\ref{tab:convergence}. We observe that, independently of the mesh size, the relative errors of the reconstructed initial data are $\approx 0.06$, while the relative errors of the reconstruction from purely acoustic time-reversal are roughly twice as large $\approx 0.12$.  We present reconstructed initial pressures for both meshes in Figure~\ref{fig:refinedPAT} and \ref{fig:coarsePAT}. From these figures, we observe that using a purely acoustic solver results in larger background noise than using a coupled acoustic-elastic solver.  We also observe that the error in the reconstruction is \note{concentrated} near the boundary of eclipses and at the elastic-acoustic interfaces. The former is due to high gradients in the solution, while the latter may be due to the retention of energy within the elastic region.

\section{Conclusion and future work}\label{sec:conclusion}

In this paper, we present a high order discontinuous Galerkin method for wave propagation in coupled elastic-acoustic media.  The method utilizes easily invertible weight-adjusted approximations of weighted mass matrices, as well as an upwind-like penalty numerical flux across the interface between elastic and acoustic media.  The formulation is provably discretely energy stable and consistent on arbitrary heterogeneous media, including anisotropy and sub-cell \note{micro-heterogeneities}.  An extension of the method to curvilinear meshes achieves similar results.  Numerical examples confirm the high order accuracy of this method for analytic solutions to classical interface problems, and results produced by the proposed method are consistent with existing results for isotropic and anisotropic heterogeneous media.  

Future work includes the acceleration of the proposed method using tailored Bernstein-Bezier algorithms \cite{chan2017gpu, kg2018bern}, which can reduce the computational complexity of the implementation from $O(N^{2d})$ to $O(N^{d+1})$ in $d$ dimensions, as well as extensions to wave propagation in acoustic-elastic-poroelastic media \cite{shukla2019weight}.  

\section*{Acknowledgments}
Kaihang Guo and Jesse Chan acknowledge the support of the National Science Foundation under awards DMS-1719818 and DMS-1712639.  The work of Sebastian Acosta was partially supported by NSF grant DMS-1712725.


\bibliographystyle{model1-num-names}
\bibliography{refs}

\end{document}